\documentclass[12pt,reqno]{amsart}
\usepackage{amssymb,amsmath,amsthm,amsfonts,mathrsfs,graphicx}
\usepackage{mathtools}
\usepackage{enumerate,enumitem}
\usepackage[margin=2.5cm]{geometry}
\usepackage[colorlinks=true,linkcolor=blue]{hyperref}
\usepackage{tikz}

\title[Fractional Primitive Equations]{Global well-posedness of the 2D primitive equations with fractional horizontal dissipation}

\author[Changhui Tan]{Changhui Tan}
\address[Changhui Tan]{\newline Department of Mathematics, University of South Carolina, 1523 Greene St, Columbia SC 29208, USA}
\email{tan@math.sc.edu}

\author[Zhuan Ye]{Zhuan Ye}
\address[Zhuan Ye]{\newline Department of Mathematics and Statistics, Jiangsu Normal University, 101 Shanghai Road, Xuzhou 221116, Jiangsu, PR China}
\email{yezhuan815@126.com}

\thanks{\textit{Acknowledgment.} 
CT is partially supported by NSF grant DMS-2238219.
}

\makeatletter
\@namedef{subjclassname@2020}{\textup{2020} Mathematics Subject Classification}
\makeatother


\newtheorem{theorem}{Theorem}[section]
\newtheorem{lemma}[theorem]{Lemma}

\theoremstyle{definition}

\theoremstyle{remark}
\newtheorem{remark}{Remark}

\def\T{\mathbb{T}}

\def\dd{\mathrm{d}}
\def\dxdz{\dd x\dd z}

\begin{document}
\allowdisplaybreaks

\begin{abstract}
In this paper, we investigate the two-dimensional incompressible primitive equations with fractional horizontal dissipation. Specifically, we establish global well-posedness of strong solutions for arbitrarily large initial data when the dissipation exponent satisfies $\alpha\geq\alpha_{0}\approx1.1108$. In addition, we prove global well-posedness of strong solutions for small initial data when $\alpha \in [1, \alpha_0)$. Notably, the smallness assumption is imposed only on the $L^\infty$ norm of the initial vorticity.
\end{abstract}

\maketitle 


\section{Introduction}\label{sec:intro}
The primitive equations constitute one of the fundamental models in planetary oceanic and atmospheric dynamics; see, for instance, the books \cite{majda2003introduction, vallis2017atmospheric}. These equations are derived from the (rotating) incompressible Navier-Stokes equations under the hydrostatic balance assumption for the vertical pressure component. In geophysical fluid dynamics, the horizontal turbulent mixing is typically much stronger than the vertical one, so that the horizontal viscosity dominates while the vertical viscosity is comparatively weak and is often neglected.

In this paper, we consider the two-dimensional primitive equations with fractional horizontal dissipation:
\begin{equation}\label{eq:main}
\left\{\begin{array}{l}
\partial_t u + u\partial_{x}u+\widetilde{w}\partial_{z}u + \partial_{x}p+\nu\Lambda_{h}^{\alpha}u= 0,\\[3pt]
\partial_{x}u+ \partial_{z}\widetilde{w}=0, \\[3pt]
\partial_{z}p=0,
\end{array}\right.
\end{equation}
defined on a 2D periodic channel
\[
 \Omega \triangleq \big\{(x,z):\ x\in \mathbb{T},\,z\in [0,1]\big\},
\]
where $(x,z)$ denote the horizontal and vertical variables, $(u,\widetilde{w})$ represent the horizontal and vertical velocity components, respectively, and $p$ is the pressure. 
Here $\mathbb{T}$ denotes the 1D periodic domain of length one, and $\Lambda_{h}^{\alpha}=(-\partial_{x}^{2})^{\frac{\alpha}{2}}$ denotes the horizontal fractional Laplacian, defined by
\[
\Lambda_{h}^{\alpha}u(x,z) \triangleq c_\alpha\,{\rm p.v.}\int_{\mathbb R} \frac{u(x,z)-u(y,z)}{|x-y|^{1+\alpha}}\,\dd y,\quad c_\alpha = \tfrac{2^\alpha\Gamma(\frac{\alpha+1}{2})}{\sqrt\pi\,|\Gamma(-\frac\alpha2)|}, \quad \alpha\in(0,2),
\]
where $u$ is viewed as a $1$-periodic function in the $x$-variable and ${\rm p.v.}$ denotes the principal value. The fractional Laplacian provides a continuous interpolation between the inviscid system (as $\alpha \to 0$) and the viscous system (as $\alpha \to 2$).
The equation \eqref{eq:main}$_2$ enforces incompressibility, while \eqref{eq:main}$_3$ expresses the hydrostatic pressure balance.

The system \eqref{eq:main} can be regarded as the hydrostatic limit of the Navier-Stokes equations with fractional horizontal dissipation, and the derivation follows analogously to \cite{azerad2001mathematical,li2019primitive}. 

We further impose the initial condition 
\[
u(x,z,0) = u_0(x,z),
\]
and the boundary condition
\[
\widetilde{w}(x,0,t)=\widetilde{w}(x,1,t)=0.
\]
Together with \eqref{eq:main}$_2$,  this implies that the vertical velocity $\widetilde{w}$ is uniquely determined by $u$ through
\[
\widetilde{w}(x,z,t)=-\int_{0}^{z}\partial_{x}u(x,\tilde{z},t)\,\dd\tilde{z}.
\]
This introduces a loss of one horizontal derivative in $\widetilde{w}$ compared to the Navier-Stokes equations, a distinctive characteristic of the primitive equations.

\vskip 1em
The first systematic mathematical studies of the primitive equations were carried out in the 1990s by Lions, Temam, and Wang \cite{lions1992equations,lions1992new}, who established the global existence of weak solutions. However, the uniqueness of weak solutions remains an open problem, even in the two-dimensional setting; see \cite{bresch2003uniqueness,kukavica2014primitive,li2017existence,medjo2010uniqueness} for several partial results in this direction.

The global existence and uniqueness of strong solutions to the three-dimensional primitive equations with full viscosity have been known since the breakthrough work of Cao and Titi \cite{cao2007global}, which exploits the anisotropic structure of the system. See also \cite{hieber2016global,kobelkov2006existence,kukavica2007regularity} for further generalizations. The well-posedness theory was later extended to the primitive equations with only horizontal viscosity in a series of works \cite{cao2016global,cao2017strong,cao2020global}.

On the other hand, for the inviscid primitive equations, the loss of a horizontal derivative in the vertical velocity renders the system ill-posed in Sobolev spaces, even locally in time; see \cite{han2016ill,renardy2009ill}. Local well-posedness can be recovered under additional structural assumptions, such as the local Rayleigh condition \cite{brenier1999homogeneous,kukavica2014local,masmoudi2012h}, or by assuming real-analytic initial data \cite{kukavica2011local}. Moreover, smooth solutions may develop singularities in finite time; see \cite{cao2015finite,collot2023stable,wong2015blowup}.

These contrasting behaviors highlight that horizontal viscosity plays a crucial role in stabilizing the flow. To understand the transition between the viscous and inviscid regimes, Abdo, Lin, and Tan \cite{abdo2025well} introduced a family of primitive equations with fractional horizontal dissipation \eqref{eq:main}, which interpolate between the two systems. They identified a sharp transition between local well-posedness and ill-posedness at the critical dissipation exponent $\alpha = 1$. Specifically, the system \eqref{eq:main} is ill-posed in Sobolev spaces when $\alpha < 1$, analogous to the inviscid case, while it is locally well-posed when $\alpha > 1$, analogous to the viscous case. In the critical regime $\alpha = 1$, the well-posedness of solutions depends on the interplay between the horizontal viscosity coefficient $\nu$ and the size of the initial data.

Furthermore, they established a Beale-Kato-Majda type regularity criterion
\begin{equation}\label{eq:BKM}	
\int_0^T\|\Lambda_{h}^{\frac{3-\alpha}{2}}\omega(t)\|_{L^2}^2\dd t<\infty,
\end{equation}
under which a local solution can be extended up to time $T$, where $\omega \triangleq \partial_z u$ denotes the hydrostatic vorticity. The criterion \eqref{eq:BKM} was verified for $\alpha\geq1$ under a smallness assumption on the initial data, leading to a global well-posedness theory for small initial data. For general initial data, global well-posedness was established for $\alpha \geq \tfrac{6}{5}$. They further posed the following intriguing open problem:
\begin{quote}
\emph{The global well-posedness problem for the system \eqref{eq:main} with general initial data in the range $\alpha \in (1, \tfrac{6}{5})$ remains open.}
\end{quote}

\vskip 1em
In this paper, we address and partially resolve this problem. In particular, we establish the global well-posedness of \eqref{eq:main} for $\alpha \geq \alpha_0$, where $\alpha_0 \approx 1.1108 < \tfrac{6}{5}$. The precise statement of our main theorem is as follows.

\begin{theorem}\label{thm:GWP}
Let $T>0$ and $\alpha\in[\alpha_{0},\,2]$, where $1.1108\approx\alpha_{0}\in (1,\frac{6}{5})$ is the root of the cubic equation $2\alpha^3+3\alpha^2-4\alpha-2=0.$

Assume that the initial data $u_0$ satisfies
\begin{align}\label{eq:init}
& u_0 \in \mathsf{H} \triangleq \Big\{ f\in L^{2}(\Omega): \int_{0}^{1}f(x,z)\,\dd z=0,\,\,\forall\,x\in \mathbb{T}\Big\}, \quad  \Lambda_{h}^{\delta_m}u_0 \in L^{2}(\Omega),\nonumber\\
& \omega_{0}\triangleq\partial_{z}u_0 \in L^{\infty}(\Omega),\quad \Lambda_{h}^{\frac{3-2\alpha}{2}}\omega_{0} \in L^{2}(\Omega),\quad \text{and}\quad \partial_{z}\omega_{0} \in L^{2}(\Omega),
\end{align}
where
\begin{equation}\label{eq:deltam}
	\delta_m \triangleq \max \Big\{\frac{\alpha}{2},\,\frac{2(-\alpha^2-\alpha+3)}{3-\alpha}\Big\}.
\end{equation}
Then, the system \eqref{eq:main} with initial data $u_0$ has a unique global strong solution on $[0,T]$, with
\begin{align}\label{eq:sol}
& u\in C\big([0,T], \mathsf{H}\big),\quad \Lambda_{h}^{\delta_m}u \in L^\infty\big([0,T], L^{2}(\Omega)\big),\quad  \Lambda_{h}^{\delta_m + \frac\alpha2}u \in L^2\big([0,T], L^2(\Omega)\big),\nonumber\\
& \omega\triangleq\partial_{z}u \in L^\infty\big([0,T], L^\infty(\Omega)\big),\quad \Lambda_{h}^{\frac{3-2\alpha}{2}}\omega \in L^\infty\big([0,T], L^{2}(\Omega)\big),\quad  \Lambda_{h}^{\frac{3-\alpha}{2}}\omega \in L^2\big([0,T], L^2(\Omega)\big),\nonumber\\
& \partial_{z}\omega \in L^\infty\big([0,T], L^{2}(\Omega)\big),\quad \text{and}\quad \Lambda_{h}^{\frac{\alpha}{2}}\partial_{z}\omega \in L^2\big([0,T], L^2(\Omega)\big).
\end{align}
\end{theorem}
\vskip 1em

\begin{remark}
A strong solution of \eqref{eq:main} is defined as (see \cite[Definition 1]{abdo2025well})
\begin{align*}
& u\in C\big([0,T], \mathsf{H}\big),\quad \Lambda_{h}^{\frac\alpha2}u \in L^\infty\big([0,T], L^{2}(\Omega)\big),\quad  \Lambda_{h}^{\alpha}u \in L^2\big([0,T], L^2(\Omega)\big),\\
& \omega \in L^\infty\big([0,T], L^\infty(\Omega)\big),\quad \quad  \Lambda_{h}^{\frac{\alpha}{2}}\omega \in L^2\big([0,T], L^2(\Omega)\big).
\end{align*}
It is known that such a strong solution may not be unique. A sufficient condition ensuring uniqueness is the regularity criterion \eqref{eq:BKM}. Theorem \ref{thm:GWP} establishes \eqref{eq:BKM} and therefore guarantees uniqueness of the strong solution.

The parameter $\delta_m$ in \eqref{eq:deltam} represents the minimal regularity required of $u_0$ for global well-posedness. Note that $\delta_m = \tfrac{\alpha}{2}$ when $\alpha \geq \tfrac{-7+\sqrt{193}}{6} \approx 1.149$. In this case, no additional regularity on $u_0$ is needed to guarantee uniqueness of the strong solution. Nonetheless, additional regularity is still required for $\omega_0$ and $\partial_z \omega_0$.

In \cite{abdo2025well}, it was shown that the regularity criterion \eqref{eq:BKM} also ensures global well-posedness for classical and smooth solutions. Consequently, if the initial data $u_0$ possess additional regularity, Theorem \ref{thm:GWP} yields global well-posedness for classical and smooth solutions.
\end{remark}

The key idea underlying the proof of Theorem \ref{thm:GWP} is to make full use of the a priori $L^\infty$ bound on the hydrostatic vorticity $\omega$ (see \eqref{eq:omegaMP}). This bound serves as the foundation for deriving a hierarchy of enhanced a priori energy estimates, which significantly strengthen the regularity control of the solution. We develop an iterative procedure to obtain successive a priori bounds at higher levels of regularity. As a result, we are able to relax the global well-posedness assumption from $\alpha \geq \tfrac{6}{5}$ to the improved threshold $\alpha \geq \alpha_0$, thereby extending the range of admissible dissipation exponents.  
 
\vskip 1em

Our next result focuses on the global well-posedness of \eqref{eq:main} for $\alpha\in[1,\alpha_0)$ under the assumption of small initial data.
\begin{theorem}\label{thm:small}
Let $T>0$ and $\alpha\in[1,\alpha_0)$. Assume that the initial data $u_0$ satisfies \eqref{eq:init}. Then, there exists a small constant $c>0$, such that if
\[ \|\omega_0\|_{L^\infty}<c, \]
then the system \eqref{eq:main} with initial data $u_0$ has a unique global strong solution on $[0,T]$, satisfying \eqref{eq:sol}.
\end{theorem}

\begin{remark}
A small-data global well-posedness result was established in \cite{abdo2025well}, requiring
\[
 \|\omega_0\|_{L^\infty}<c,\quad\text{and}\quad 
 \|\Lambda_{h}^\alpha u_{0}\|^2_{L^2}+\|\Lambda_{h}^{\frac{\alpha}{2}}\omega_{0}\|_{L^{2}}^{2}+ \|\partial_{z}\omega_{0}\|^2_{L^2}<c.
\]
In contrast, our result does not impose any smallness condition on $\|\Lambda_{h}^\alpha u_{0}\|^2_{L^2}+\|\Lambda_{h}^{\frac{\alpha}{2}}\omega_{0}\|_{L^{2}}^{2}+ \|\partial_{z}\omega_{0}\|^2_{L^2}$. In fact, this quantity need not even be bounded. 
Under the assumption \eqref{eq:init}, we only require
\[
X_0 = \|\Lambda_{h}^{\delta_m} u_{0}\|^2_{L^2}+\|\Lambda_{h}^{\frac{3-2\alpha}{2}}\omega_{0}\|_{L^{2}}^{2}+ \|\partial_{z}\omega_{0}\|^2_{L^2}<\infty.
\]
The small constant c appearing in \eqref{eq:c} is given explicitly and may depend on $\nu$ and $X_0$.
Crucially, however, smallness of $X_0$ itself is not required.
Hence, our result substantially relaxes the smallness assumption compared to \cite{abdo2025well}.

Moreover, our result extends to the critical regime, when $\alpha = 1$.
Together with the ill-posedness result for $\|\omega_0\|_{L^\infty} \gg \nu$ in the critical regime (see \cite[Theorem 5.2]{abdo2025well}), this demonstrates a sharp transition in the solution behavior: from ill-posedness to well-posedness.
\end{remark}

The proof of Theorem \ref{thm:small} utilizes the same a priori energy estimates used in the proof of Theorem \ref{thm:GWP}. Consequently, the regularity assumptions on the initial data are consistent across the two results.

\vskip 1em
The rest of the paper is organized as follows. In Section \ref{sec:prelim}, we present preliminary results, including several a priori bounds and functional inequalities. Section \ref{sec:apriori} is devoted to establishing enhanced a priori energy estimates. Building on these estimates, we prove our main results, Theorems \ref{thm:GWP} and \ref{thm:small}, in Sections \ref{sec:GWP} and \ref{sec:GWPsmall}, respectively.

\subsection*{Notations}
We use the symbol $C$ to denote a generic positive constant, which may vary from line to line. The dependence of $C$ on specific parameters will be clear from the context and explicitly indicated whenever necessary.

We write $L^p_x = L^p_x(\T)$ and $L^p_z = L^p_z([0,1])$ for the Lebesgue spaces in the $x$- and $z$-variables, respectively.
For a function $f = f(x,z)$, we define the mixed norm
\[
\|f\|_{L_{z}^{q}L_{x}^{p}}\triangleq \Big\|\|f\|_{L_{x}^{p}}\Big\|_{L_{z}^{q}}
= \bigg(\int_0^1\Big(\int_\T f(x,z)^p\,\dd x\Big)^{\frac{q}{p}}\,\dd z\bigg)^{\frac{1}{q}},
\]
and denote $L^p = L^p(\Omega) = L^p_zL^p_x$.


\section{Preliminaries}\label{sec:prelim}
\subsection{A priori bounds}
We collect several useful \textit{a priori} bounds for the solution of \eqref{eq:main}, with the proofs given in \cite{abdo2025well}.

We start with the basic $L^2$ bound on $u$:
\begin{equation}\label{eq:uL2}
 \|u(t)\|^2_{L^2} + \nu\int_0^t\|\Lambda_{h}^{\frac{\alpha}{2}}u(\tau)\|_{L^2}^2  \dd\tau\leq
\|u_{0}\|^2_{L^2}.
\end{equation}

Differentiating \eqref{eq:main}$_1$ with respect to $z$, we obtain the evolution equation for the hydrostatic vorticity $\omega = \partial_zu$:
\begin{equation}\label{eq:omega}
\partial_t \omega + (u\partial_{x}+\widetilde{w}\partial_{z})\omega +\nu\Lambda_{h}^{\alpha}\omega= 0.
\end{equation}
An energy estimate for \eqref{eq:omega} gives a same type of $L^2$ bound as \eqref{eq:uL2}:
\begin{equation}\label{eq:omegaL2}
\|\omega(t)\|^2_{L^2}
   + \nu\int_0^t\|\Lambda_{h}^{\frac{\alpha}{2}}\omega(\tau)\|_{L^2}^2  \dd\tau\leq
\|\omega_{0}\|^2_{L^2}.
\end{equation}
Moreover, $\omega$ satisfies the following maximum principle:
\begin{equation}\label{eq:omegaMP}
\|\omega(t)\|_{L^{\infty}} \leq\|\omega_{0}\|_{L^{\infty}}.
\end{equation}

Thanks to the fact that $\int_{0}^{1}u(x,z,t)\,dz=0$ (see \cite[Lemma 2.1]{abdo2025well}), we apply Poincar\'e inequality in $z$-direction and obtain
\begin{equation}\label{eq:uLinf}
  \|u(t)\|_{L^{\infty}}\leq C\|\partial_zu(t)\|_{L^{\infty}} \leq C\|\omega_0\|_{L^{\infty}}.
\end{equation}
Applying the same Poincar\'e inequality also yields control of the vertical velocity $\widetilde{w}$:
\begin{equation}\label{eq:wPoincare}
  \|\widetilde{w}\|_{L^p_z}\leq C\|\partial_z \widetilde{w}\|_{L^p_z} = C\|\partial_x u\|_{L^p_z},\quad \forall\,1\leq p\leq \infty.
\end{equation}
We will also make use of the following anisotropic estimate, which provides stronger control of $\widetilde{w}$:
\begin{equation}\label{eq:wAniso}
 \|\widetilde{w}\|_{L^\infty_z L^p_x}\leq C	\|\widetilde{w}\|_{L^p_x L^\infty_z} \leq C \|\partial_xu\|_{L^p_x L^2_z}\leq C\|\partial_xu\|_{L^2_z L^p_x},\quad\forall\,2\leq p<\infty,
\end{equation}
where we have used Minkowski's inequality (twice) and Sobolev embedding in the intermediate step.

Next, for $\alpha\in(1,2]$,
\begin{equation}\label{eq:uE1}
\|\Lambda_{h}^{\frac{\alpha}{2}}u(t)\|^2_{L^2}
   + \nu\int_0^t\|\Lambda_{h}^{\alpha}u(\tau)\|_{L^2}^2  \dd\tau\leq C,
\end{equation}
where $C$ depends on $\|\Lambda_{h}^{\frac{\alpha}{2}}u_0\|_{L^2}$, $\|\omega_0\|_{L^\infty}$ and $t$. 

\subsection{Useful inequalities} We list the following classical functional inequalities which will be used in our analysis.
\begin{lemma}[Fractional Leibniz rule]\label{lem:Leibniz}
 Let $s>0$. Assume that 
 \[\frac{1}{r}=\frac{1}{p_1}+\frac{1}{q_1}=\frac{1}{p_2}+\frac{1}{q_2},\quad \text{where}\quad r,\, q_1,\, p_2\in(1,\infty),\quad\text{and}\quad p_1,\, q_2\in[1,\infty].\]
Then we have
\begin{equation}\label{eq:Leib}
 \|\Lambda_{h}^s(fg)\|_{L_{x}^r}
  \leq C\Big(\|f\|_{L_{x}^{p_1}}\|\Lambda_{h}^{s}g \|_{L_{x}^{q_1}} + \|\Lambda_{h}^{s} f \|_{L_{x}^{p_2}} \| g\|_{L_{x}^{q_2}}\Big),
\end{equation}
where $C$ depends on the parameters $s, r, p_1, q_1, p_2$ and $q_2$.	
\end{lemma}

\begin{lemma}[Interpolations]\label{lem:interpolation}
 The following two types of interpolation inequalities hold:
 \begin{enumerate}
 \item Let $s_1\leq s \leq s_2$. Then
 \begin{equation}\label{eq:interp1}
\|\Lambda_h^{s}	f\|_{L^2_x}\leq \|\Lambda_h^{s_1}	f\|_{L^2_x}^{1-\theta}\|\Lambda_h^{s_2}	f\|_{L^2_x}^\theta,\quad \text{where}\quad \theta=\frac{s-s_1}{s_2-s_1}. 	
 \end{equation}
 \item Let $0\leq j<m$ and $p=\frac{2m}{j}$. Then
 \begin{equation}\label{eq:interp2}
\|\Lambda_h^{j}\omega\|_{L^p_x}\leq \|\omega\|_{L^\infty_x}^{1-\theta}\|\Lambda_h^{m}\omega\|_{L^2_x}^\theta,\quad \text{where}\quad \theta=\frac{2}{p}=\frac{j}{m}. 	
 \end{equation}
 \end{enumerate}
\end{lemma}

\section{Enhanced a priori energy bounds}\label{sec:apriori}
In this section, we develop a series of enhanced a priori energy estimates that play a crucial role in our analysis. The central idea is to exploit the a priori bound on  $\|\omega(t)\|_{L^\infty}$ given in \eqref{eq:omegaMP} and, building on this, to derive enhanced energy bounds for higher-order Sobolev norms.

We first present an improved a priori regularity estimate for $u$.
\begin{lemma}\label{lem:uimprove}
Let $T>0$ and $\alpha\in(1,\frac{6}{5})$. Then we have
\begin{equation}\label{eq:uimprove}
\|\Lambda_{h}^{\delta}u(t)\|^2_{L^2} + \nu\int_0^t\|\Lambda_{h}^{\delta+\frac{\alpha}{2}}u(\tau)\|_{L^2}^2  \dd\tau\leq C,\quad\forall\, t\in[0,T],
\end{equation}
for any $\delta\in [0,\delta_1]$, where
\begin{align}\label{eq:delta1}
\delta_1 \triangleq \frac{2\alpha^{2}-\alpha}{2},
\end{align}
and $C$ depends on $\|\Lambda_{h}^{\delta}u_0\|_{L^2}$, $\|\omega_0\|_{L^\infty}$ and $T$.
\end{lemma}

\begin{remark}
Since $\delta_1 > \tfrac{\alpha}{2}$ whenever $\alpha > 1$, Lemma \ref{lem:uimprove} yields an enhanced energy bound compared to \eqref{eq:uE1}.
\end{remark}

\begin{proof}[Proof of Lemma \ref{lem:uimprove}]
For $\delta\in[0,\frac\alpha2]$, \eqref{eq:uimprove} follows from \eqref{eq:uL2} and \eqref{eq:uE1}. In the following, we assume $\delta>\frac\alpha2$.

Taking the $L^2$ inner product of equation $\eqref{eq:main}_{1}$ with $\Lambda_{h}^{2\delta}u$, it yields
\begin{align}\label{eq:udelta}
 \frac{1}{2}\frac{\dd}{\dd t}
  \|\Lambda_{h}^{\delta}u\|^2_{L^2}
  +\nu \|\Lambda_{h}^{\delta+\frac{\alpha}{2}}u\|_{L^2}^2 &= - \int_{\Omega}u\,\partial_{x}u\, \Lambda_{h}^{2\delta}u\, \dxdz-\int_{\Omega}\widetilde{w}\,\omega\, \Lambda_{h}^{2\delta}u\, \dxdz\triangleq U_{1}+U_{2}.
\end{align}

For $U_1$, we apply fractional Leibniz rule \eqref{eq:Leib}, \eqref{eq:uLinf}, interpolation \eqref{eq:interp1}, and Young's inequality to get
\begin{align}\label{eq:U1}
U_{1} 
&= -\frac{1}{2}\int_{\Omega}\Lambda_{h}^{\delta-\frac{\alpha}{2}}\partial_{x}(u^2)\, \Lambda_{h}^{\delta+\frac{\alpha}{2}}u \,\dxdz
\leq C\|\Lambda_{h}^{\delta-\frac{\alpha}{2}+1}(u^2)\|_{L^2} \|\Lambda_{h}^{\delta+\frac{\alpha}{2}}u\|_{L^2}\nonumber\\
& \leq C\|u\|_{L^{\infty}}\|\Lambda_{h}^{\delta-\frac{\alpha}{2}+1}u\|_{L^2}\|\Lambda_{h}^{\delta+\frac{\alpha}{2}}u\|_{L^2}
\leq C\|\omega_0\|_{L^{\infty}}\|\Lambda_{h}^{\delta}u\|_{L^2}^{\frac{2\alpha-2}{\alpha}} \|\Lambda_{h}^{\delta+\frac{\alpha}{2}}u\|_{L^2}^{\frac{2}{\alpha}}  \nonumber\\
&\leq \frac{\nu}{4} \|\Lambda_{h}^{\delta+\frac{\alpha}{2}}u\|_{L^2}^2 + C\|\omega_0\|_{L^{\infty}}^{\frac{\alpha}{\alpha-1}} \|\Lambda_{h}^{\delta} u\|_{L^2}^2.
\end{align}

For $U_{2}$, we apply fractional Leibniz rule \eqref{eq:Leib} and obtain
\begin{align}\label{eq:U2}
U_{2} & = - \int_{\Omega}\Lambda_{h}^{\delta-\frac{\alpha}{2}}(\widetilde{w}\omega)\,\Lambda_{h}^{\delta+\frac{\alpha}{2}}u \,\dxdz
\leq C\|\Lambda_{h}^{\delta-\frac{\alpha}{2}}(\widetilde{w}\omega)\|_{L^2} \|\Lambda_{h}^{\delta+\frac{\alpha}{2}}u\|_{L^2}. \nonumber\\
 & \leq C\Big\|\|\omega\|_{L_{x}^{\infty}}\|\Lambda_{h}^{\delta-\frac{\alpha}{2}} \widetilde{w}\|_{L_{x}^2}+\|\widetilde{w}\|_{L_{x}^{\frac{2p}{p-2}}} \|\Lambda_{h}^{\delta-\frac{\alpha}{2}}\omega\|_{L_{x}^{p}}\Big\|_{L^2_z}\|\Lambda_{h}^{\delta+\frac{\alpha}{2}}u\|_{L^2} \nonumber\\
 & \leq C \Big(\|\omega\|_{L^{\infty}}\|\Lambda_{h}^{\delta-\frac{\alpha}{2}} \widetilde{w}\|_{L^2} + \|\widetilde{w}\|_{L_{z}^{\infty}L_{x}^{\frac{2p}{p-2}}} \|\Lambda_{h}^{\delta-\frac{\alpha}{2}}\omega\|_{L_{z}^{2}L_{x}^{p}}\Big)\|\Lambda_{h}^{\delta+\frac{\alpha}{2}}u\|_{L^2} \nonumber\\
 & \leq C \Big(\|\omega_0\|_{L^{\infty}}\|\Lambda_{h}^{\delta-\frac{\alpha}{2}} \partial_xu\|_{L^2} + \|\partial_xu\|_{L_{z}^{2}L_{x}^{\frac{2p}{p-2}}} \|\Lambda_{h}^{\delta-\frac{\alpha}{2}}\omega\|_{L_{z}^{2}L_{x}^{p}}\Big)\|\Lambda_{h}^{\delta+\frac{\alpha}{2}}u\|_{L^2}
 \triangleq U_{21} + U_{22},
\end{align}
where the parameter $p\in(2,\infty)$ will be chosen later in \eqref{eq:p}.
Note that in the last inequality, we have used \eqref{eq:omegaMP} and \eqref{eq:wPoincare} for the first term, and the anisotropic estimate \eqref{eq:wAniso} for the second term.

The term $U_{21}$ can be estimated by interpolation \eqref{eq:interp1} as follows:
\begin{equation}\label{eq:U21}
 U_{21} \leq C \|\omega_0\|_{L^\infty} \|\Lambda_{h}^{\delta} u\|_{L^2}^{\frac{2\alpha-2}{\alpha}}\|\Lambda_{h}^{\delta+\frac\alpha2} u\|_{L^2}^{\frac2\alpha}
 \leq \frac{\nu}{8} \|\Lambda_{h}^{\delta+\frac{\alpha}{2}}u\|_{L^2}^2 + C\|\omega\|_{L^{\infty}}^{\frac{\alpha}{\alpha-1}}\|\Lambda_{h}^{\delta} u\|_{L^2}^2.
\end{equation}

For $U_{22}$, we apply Sobolev embedding and interpolation and get
\begin{equation}\label{eq:ux}
\|\partial_{x}u\|_{L_{z}^{2}L_{x}^{\frac{2p}{p-2}}} \leq  C\|\Lambda_{h}^{1+\frac{1}{p}}u\|_{L^{2}}
\leq C\|\Lambda_{h}^{\delta}u\|_{L^{2}}^{1-\theta_{11}} \|\Lambda_{h}^{\delta+\frac{\alpha}{2}}u\|_{L^{2}}^{\theta_{11}},\quad \text{where}\quad \theta_{11} = \tfrac{2-2\delta+\frac{2}{p}}{\alpha}.
\end{equation}
The main enhancement comes from the estimate of $\|\Lambda_{h}^{\delta-\frac{\alpha}{2}}\omega\|_{L^p_x}$. Thanks to the a priori bounds \eqref{eq:omegaMP} and \eqref{eq:omegaL2},  we use \eqref{eq:interp2} and interpolate this term by $\|\omega\|_{L^\infty_x}$ and $\|\Lambda_{h}^{\frac{\alpha}{2}}\omega\|_{L^2_x}$. 
Choosing 
\begin{equation}\label{eq:p}
  p \triangleq \frac{2\alpha}{2\delta-\alpha} \in (2,\infty),\quad \text{so that} \quad
  \theta_{12}=\frac{2}{p} = \frac{\delta-\frac\alpha2}{\frac\alpha2},
\end{equation}
we have
\begin{equation}\label{eq:omegaenhance}
  \|\Lambda_{h}^{\delta-\frac{\alpha}{2}}\omega\|_{L^p_x}\leq C \|\omega\|_{L^\infty_x}^{1-\theta_{12}}\|\Lambda_{h}^{\frac{\alpha}{2}}\omega\|_{L^2_x}^{\theta_{12}}
  = C \|\omega\|_{L^\infty_x}^{\frac{2\alpha-2\delta}{\alpha}}\|\Lambda_{h}^{\frac{\alpha}{2}}\omega\|_{L^2_x}^{\frac{2\delta-\alpha}{\alpha}}.
\end{equation}
Consequently, we deduce
\begin{align}\label{eq:U22}
 U_{22} & \leq C \|\omega_0\|_{L^\infty}^{1-\theta_{12}}\|\Lambda_{h}^{\frac{\alpha}{2}}\omega\|_{L^2}^{\theta_{12}}\|\Lambda_{h}^{\delta}u\|_{L^{2}}^{1-\theta_{11}} \|\Lambda_{h}^{\delta+\frac{\alpha}{2}}u\|_{L^{2}}^{1+\theta_{11}} \nonumber\\
 &\leq \frac{\nu}{8} \|\Lambda_{h}^{\delta+\frac{\alpha}{2}}u\|_{L^2}^2 + C\|\omega_0\|_{L^{\infty}}^{\frac{2(1-\theta_{12})}{1-\theta_{11}}}\|\Lambda_{h}^{\frac{\alpha}{2}}\omega\|_{L^{2}}^{\frac{2\theta_{12}}{1-\theta_{11}}}\|\Lambda_{h}^{\delta}u\|_{L^{2}}^{2}.
\end{align}
Since we only have the a priori control $\|\Lambda_{h}^{\frac{\alpha}{2}}\omega\|_{L^2}\in L^2(0,T)$ from \eqref{eq:omegaL2}, in order to apply Gr\"onwall inequality, we require
\begin{equation}\label{eq:deltabound}
 \frac{2\theta_{12}}{1-\theta_{11}}=\frac{2\alpha(2\delta-\alpha)}{(\alpha-1)(2\delta+\alpha)}\leq 2,\quad\Longleftrightarrow\quad
 \delta\leq \frac{2\alpha^2-\alpha}{2}\triangleq \delta_1.
\end{equation}
Hence, the largest $\delta$ we can choose is $\delta_1$ in \eqref{eq:delta1}.

Inserting the estimates \eqref{eq:U1}, \eqref{eq:U21} and \eqref{eq:U22} into \eqref{eq:udelta}, it follows that
\[
 \frac{\dd}{\dd t}\|\Lambda_{h}^{\delta_1}u\|^2_{L^2} + \nu \|\Lambda_{h}^{\delta_1+\frac{\alpha}{2}}u\|_{L^2}^2
 \leq C \Big(\|\omega_0\|_{L^{\infty}}^{\frac{\alpha}{\alpha-1}} + C\|\omega_0\|_{L^{\infty}}^{\frac{2(1-\theta_{12})}{1-\theta_{11}}}\|\Lambda_{h}^{\frac{\alpha}{2}}\omega\|_{L^{2}}^{\frac{2\theta_{12}}{1-\theta_{11}}}\Big)\|\Lambda_{h}^{\delta_1} u\|_{L^2}^2.
\]
Applying Gr\"onwall inequality and \eqref{eq:omegaL2}, we obtain that for any $\delta\in[\frac\alpha2,\delta_1]$,
\begin{align*}
 \|\Lambda_{h}^{\delta}u(t)\|^2_{L^2} + \nu\int_0^t\|\Lambda_{h}^{\delta+\frac{\alpha}{2}}u(\tau)\|_{L^2}^2 \dd \tau
 &\leq \|\Lambda_{h}^{\delta}u_0\|^2_{L^2}\exp\Big(C\int_0^t \big(1 + \|\Lambda_{h}^{\frac{\alpha}{2}}\omega\|_{L^{2}}^{2}\big) \dd t\Big)\\
 &\leq \|\Lambda_{h}^{\delta}u_0\|^2_{L^2}\exp\big(C(T + 1)\big), 
\end{align*}
where the constant $C$ depends on $\|\omega_0\|_{L^\infty}$. This completes the proof of \eqref{eq:uimprove}.
\end{proof}

The extra regularity that we gain from Lemma \ref{lem:uimprove} allows us to establish an a priori $L^2$-bound of $\partial_{z}\omega$ over a wider range of $\alpha\in[\alpha_0,2)$, where $\alpha_0\approx 1.1108$ is a root of the  cubic equation
\[2\alpha^3+3\alpha^2-4\alpha-2=0.\]
This improvement plays a crucial role in broadening the global regularity assumption from $\alpha \geq \tfrac{6}{5}$ in \cite{abdo2025well}. 

\begin{lemma}\label{lem:omegazimprove}
Let $T>0$ and $\alpha\in [\alpha_{0},\frac{6}{5})$. Then we have
\begin{equation}\label{eq:omegazimprove}
\|\partial_{z}\omega(t)\|^2_{L^2} + \nu \int_0^t\|\Lambda_{h}^{\frac{\alpha}{2}}\partial_{z}\omega(\tau)\|_{L^2}^2  \dd\tau\leq C,\quad \forall\, t\in[0,T],
\end{equation}
where $C$ depends on $\|\partial_z\omega_0\|_{L^2}$, $\|\Lambda_{h}^{\delta_*}u_0\|_{L^2}$, $\|\omega_0\|_{L^\infty}$, and $T$, with
\begin{equation}\label{eq:deltas}
\delta_* \triangleq \max\Big\{\frac\alpha2,\,\frac{-2\alpha^2+2\alpha+1}{\alpha}\Big\}.	
\end{equation}
\end{lemma}

\begin{proof}
Differentiating \eqref{eq:omega} in $z$ yields
\[
\partial_t \partial_{z}\omega + (u\partial_{x}+\widetilde{w}\partial_{z})\partial_{z}\omega + \nu\Lambda_{h}^{\alpha}\partial_{z}\omega=\partial_{x}u\partial_{z}\omega
-\omega\partial_{x}\omega.
\]
Taking $L^2$ inner product with $\partial_{z}\omega$, we have
\begin{equation}\label{eq:omegazL2}
 \frac{1}{2}\frac{\dd}{\dd t} \|\partial_{z}\omega\|^2_{L^2} +\nu \|\Lambda_{h}^{\frac{\alpha}{2}}\partial_{z}\omega\|_{L^2}^2 
 = - \int_{\Omega}\omega\partial_{x}\omega \partial_{z}\omega \dxdz+\int_{\Omega}\partial_{x}u (\partial_{z}\omega)^2 \dxdz \triangleq J_{1}+J_{2}.
\end{equation}

For $J_1$, we apply fractional Leibniz rule \eqref{eq:Leib}, \eqref{eq:omegaMP}, Poincar\'e inequality, and Young's inequality to get
\begin{align}\label{eq:J1}
J_{1} & = - \frac{1}{2}\int_{\Omega} \partial_{x}(\omega^{2}) \partial_{z}\omega \dxdz
\leq C\|\Lambda_{h}^{1-\frac{\alpha}{2}}(\omega^2)\|_{L^2} \|\Lambda_{h}^{\frac{\alpha}{2}}\partial_{z}\omega\|_{L^2}
\leq C\|\omega\|_{L^{\infty}}\|\Lambda_{h}^{1-\frac{\alpha}{2}}\omega\|_{L^2} \|\Lambda_{h}^{\frac{\alpha}{2}}\partial_{z}\omega\|_{L^2} \nonumber\\
&\leq C\|\omega_0\|_{L^{\infty}} \|\Lambda_{h}^{\frac{\alpha}{2}}\omega\|_{L^2} \|\Lambda_{h}^{\frac{\alpha}{2}}\partial_{z}\omega\|_{L^2}
\leq \frac{\nu}{4} \|\Lambda_{h}^{\frac{\alpha}{2}}\partial_{z}\omega\|_{L^2}^2
+C\|\omega_0\|_{L^{\infty}}^{2}\|\Lambda_{h}^{\frac{\alpha}{2}}\omega\|_{L^2}
^{2}.
\end{align}

For $J_2$, we use interpolation in $z$ and H\"older inequality in $x$ to obtain
\begin{align}\label{eq:J2-1}
J_{2}&\leq C\int_{\mathbb{T}}\|\Lambda_{h}^{1-\frac{\alpha}{2}}Hu\|_{L_{z}^{\infty}} \|\Lambda_{h}^{\frac{\alpha}{2}}\big( (\partial_{z}\omega)^2\big)\|_{L_{z}^{1}}\dd x
\nonumber\\
&\leq C\int_{\mathbb{T}}\|\Lambda_{h}^{1-\frac{\alpha}{2}}Hu\|_{L_{z}^{2}}^{\frac{1}{2}} \|\Lambda_{h}^{1-\frac{\alpha}{2}}H\partial_{z}u\|_{L_{z}^{2}}^{\frac{1}{2}} \|\Lambda_{h}^{\frac{\alpha}{2}}\big( (\partial_{z}\omega)^2\big)\|_{L_{z}^{1}} \dd x
\nonumber\\
&\leq C\int_{\mathbb{T}}\|\Lambda_{h}^{1-\frac{\alpha}{2}}Hu\|_{L_{z}^{2}}^{\frac{1}{2}} \|\Lambda_{h}^{1-\frac{\alpha}{2}}H\partial_{z}u\|_{L_{z}^{2}}^{\frac{1}{2}} \|\Lambda_{h}^{\frac{\alpha}{2}}\big( (\partial_{z}\omega)^2\big)\|_{L_{z}^{1}} \dd x
\nonumber\\
&\leq
 C \Big\|\|\Lambda_{h}^{1-\frac{\alpha}{2}}Hu\|_{L_{z}^{2}}^{\frac{1}{2}}\Big\|_{L_{x}^{2p_{1}}} \Big\|\|\Lambda_{h}^{1-\frac{\alpha}{2}}H\partial_{z}u\|_{L_{z}^{2}}^{\frac{1}{2}}\Big\|_{L_{x}^{2p_{2}}} \Big\|\|\Lambda_{h}^{\frac{\alpha}{2}}\big( (\partial_{z}\omega)^2\big)\|_{L_{z}^{1}}\Big\|_{L_{x}^{p_3}} \nonumber\\
& \leq C \|\Lambda_{h}^{1-\frac{\alpha}{2}}u\|_{L^{p_1}_x L^2_z}^{\frac{1}{2}} \|\Lambda_{h}^{1-\frac{\alpha}{2}}\omega\|_{L^{p_2}_x L^2_z}^{\frac{1}{2}} \|\Lambda_{h}^{\frac{\alpha}{2}}\big( (\partial_{z}\omega)^2\big)\|_{L^{p_3}_x L^1_z},
\end{align}
where $H=-\partial_{x}\Lambda_{h}^{-1}$ is the Hilbert transform in $x$-variable and the parameters $(p_1,p_2,p_3)$ are required to satisfy $\frac{1}{2p_1} + \frac{1}{2p_2} + \frac{1}{p_3}=1$. We proceed with the following choices of these parameters:
\[
 \frac{1}{p_1}=\frac{3-\alpha}{2}-\delta, \quad 
 \frac{1}{p_2}=\frac{2-\alpha}{2\alpha}, \quad\text{and}\quad
 \frac{1}{p_3}=\frac{2\delta\alpha+\alpha^{2}+2\alpha-2}{4\alpha}, 
\]
with $\delta\in[\frac\alpha2, \delta_1]$. 
Note that for $\alpha\in(1,\frac65)$, we have $p_1\in(2,\infty)$, $p_2\in(2,\infty)$ and $p_3\in(1,2)$.

Now, we interpolate the three terms in \eqref{eq:J2-1} one by one. 

For $\|\Lambda_{h}^{1-\frac{\alpha}{2}}u\|_{L^{p_1}_x L^2_z}$, we apply Minkowski inequality and Sobolev embedding to get
\begin{align*}
 \|\Lambda_{h}^{1-\frac{\alpha}{2}}u\|_{L^{p_1}_x L^2_z} \leq C \|\Lambda_{h}^{1-\frac{\alpha}{2}}u\|_{L^2_zL^{p_1}_x} \leq C\|\Lambda_{h}^{\delta}u\|_{L^2}\leq C.
\end{align*}
Taking advantage of the enhanced estimate \eqref{eq:uimprove} in Lemma \ref{lem:uimprove}, we know that this term is a priori bounded. Our choice of $p_1$ is based on this argument.

For $\|\Lambda_{h}^{1-\frac{\alpha}{2}}\omega\|_{L^{p_2}_x L^2_z}$, the same interpolation as \eqref{eq:omegaenhance} (replacing $\delta$ by $1$) yields
\begin{equation}\label{eq:omegaomegaz}
  \|\Lambda_{h}^{1-\frac{\alpha}{2}}\omega\|_{L^{p_2}_x}\leq C \|\omega\|_{L^\infty_x}^{1-\theta_{32}}\|\Lambda_{h}^{\frac{\alpha}{2}}\omega\|_{L^2_x}^{\theta_{32}} 
  = C \|\omega\|_{L^\infty_x}^{\frac{2\alpha-2}{\alpha}}\|\Lambda_{h}^{\frac{\alpha}{2}}\omega\|_{L^2_x}^{\frac{2-\alpha}{\alpha}},
\end{equation}
where $p_2$ is chosen such that
\[
  \theta_{32}=\frac{2}{p_2} = \frac{1-\frac\alpha2}{\frac\alpha2}.
\]
Thus,
\[
 \|\Lambda_{h}^{1-\frac{\alpha}{2}}\omega\|_{L^{p_2}_x L^2_z}\leq C \|\Lambda_{h}^{1-\frac{\alpha}{2}}\omega\|_{L^2_zL^{p_2}_x} \leq C \|\omega\|_{L^2_xL^\infty_x}^{\frac{2\alpha-2}{\alpha}}\|\Lambda_{h}^{\frac{\alpha}{2}}\omega\|_{L^2}^{\frac{2-\alpha}{\alpha}} \leq C \|\omega_0\|_{L^\infty}^{\frac{2\alpha-2}{\alpha}}\|\Lambda_{h}^{\frac{\alpha}{2}}\omega\|_{L^2}^{\frac{2-\alpha}{\alpha}}.
\]

Finally, for $\|\Lambda_{h}^{\frac{\alpha}{2}}\big( (\partial_{z}\omega)^2\big)\|_{L^{p_3}_x L^1_z}$, using fractional Leibniz rule \eqref{eq:Leib} and interpolation, it yields
\begin{align*}
 \|\Lambda_{h}^{\frac{\alpha}{2}}\big( (\partial_{z}\omega)^2\big)\|_{L^{p_3}_x}
 & \leq C\|\partial_z\omega\|_{L^{\frac{2p_3}{2-p_3}}_x}\|\Lambda_{h}^{\frac{\alpha}{2}}\partial_z\omega\|_{L^2_x}
 = C\|\partial_z\omega\|_{L^{\frac{4\alpha}{2\delta\alpha+\alpha^2-2}}_x}\|\Lambda_{h}^{\frac{\alpha}{2}}\partial_z\omega\|_{L^2_x}\\
 &\leq C\|\partial_z\omega\|_{L^2_x}^{1-\theta_{33}} \|\Lambda_{h}^{\frac{\alpha}{2}}\partial_z\omega\|_{L^2_x}^{1+\theta_{33}},
 \quad\text{where}\quad \theta_{33} = \frac{-2\delta\alpha-\alpha^2+2\alpha+2}{2\alpha^2}.
\end{align*}
Consequently, we have
\begin{align*}
 \|\Lambda_{h}^{\frac{\alpha}{2}}\big( (\partial_{z}\omega)^2\big)\|_{L^{p_3}_x L^1_z}
 & \leq C\|\Lambda_{h}^{\frac{\alpha}{2}}\big( (\partial_{z}\omega)^2\big)\|_{L^1_z L^{p_3}_x}
 \leq C\|\partial_z\omega\|_{L^2}^{1-\theta_{33}} \|\Lambda_{h}^{\frac{\alpha}{2}}\partial_z\omega\|_{L^2}^{1+\theta_{33}}.	
\end{align*}
Collecting all the estimates, we deduce that
\begin{align}\label{eq:J2}
J_2 &\leq C \|\Lambda_{h}^{\delta}u\|_{L^2}^{\frac12} \|\omega_0\|_{L^\infty}^{\frac{1-\theta_{32}}{2}}\|\Lambda_{h}^{\frac{\alpha}{2}}\omega\|_{L^2}^{\frac{\theta_{32}}{2}} \|\partial_z\omega\|_{L^2}^{1-\theta_{33}} \|\Lambda_{h}^{\frac{\alpha}{2}}\partial_z\omega\|_{L^2}^{1+\theta_{33}} \nonumber\\
& \leq \frac{\nu}{4} \|\Lambda_{h}^{\frac{\alpha}{2}}\partial_{z}\omega\|_{L^2}^2
+C\|\omega_0\|_{L^\infty}^{\frac{1-\theta_{32}}{1-\theta_{33}}}\|\Lambda_{h}^{\delta}u\|_{L^2}^{\frac{1}{1-\theta_{33}}}\|\Lambda_{h}^{\frac{\alpha}{2}}\omega\|_{L^2}^{\frac{\theta_{32}}{1-\theta_{33}}} \|\partial_z\omega\|_{L^2}^{2}.	
\end{align}
Note that $\|\Lambda_{h}^{\delta}u\|_{L^2}\in L^\infty(0,T)$ is a priori bounded by \eqref{eq:uimprove}. Therefore, we may absorb this term (as well as $\|\omega_0\|_{L^\infty}$) to the constant $C$.

Inserting the estimates \eqref{eq:J1} and \eqref{eq:J2} into \eqref{eq:omegazL2}, it follows that
\[
 \frac{\dd}{\dd t} \|\partial_{z}\omega\|^2_{L^2} +\nu \|\Lambda_{h}^{\frac{\alpha}{2}}\partial_{z}\omega\|_{L^2}^2
 \leq C\|\Lambda_{h}^{\frac{\alpha}{2}}\omega\|_{L^2}^{2} + C\|\Lambda_{h}^{\frac{\alpha}{2}}\omega\|_{L^2}^{\frac{\theta_{32}}{1-\theta_{33}}} \|\partial_z\omega\|_{L^2}^{2},
\]
where $C$ depends on $\|\Lambda_{h}^{\delta}u_0\|_{L^2}$, $\|\omega_0\|_{L^\infty}$, and $T$.

Since we only have the a priori control $\|\Lambda_{h}^{\frac{\alpha}{2}}\omega\|_{L^2}\in L^2(0,T)$ from \eqref{eq:omegaL2}, the application of Gr\"onwall inequality necessitates
\begin{equation}\label{eq:deltalow}
 \frac{\theta_{32}}{1-\theta_{33}} = \frac{2\alpha(2-\alpha)}{2\delta\alpha+3\alpha^2-2\alpha-2}\leq 2,\quad\Longleftrightarrow\quad
 \delta\geq \frac{-2\alpha^2+2\alpha+1}{\alpha}.
\end{equation}
Hence, the a priori bound on $\|\Lambda_{h}^{\delta_*}u\|_{L^2}$ is the minimal requirement, where $\delta_*$ is defined in \eqref{eq:deltas}. The bound can be obtained by applying Lemma \ref{lem:uimprove}, provided that $\delta_* \leq \delta_1$, i.e.
\[
  \frac{-2\alpha^2+2\alpha+1}{\alpha}\leq \frac{2\alpha^2-\alpha}{2},\quad\Longleftrightarrow\quad
  2\alpha^3+3\alpha^2-4\alpha-2\geq 0.
\]
The inequality holds when $\alpha\in[\alpha_0, \frac65)$. 

Then, Gr\"onwall inequality implies
\begin{align*}
 \|\partial_{z}\omega(t)\|^2_{L^2} + \nu \int_0^t\|\Lambda_{h}^{\frac{\alpha}{2}}\partial_{z}\omega(\tau)\|_{L^2}^2  d\tau
 & \leq\big(\|\partial_{z}\omega_0\|^2_{L^2}+1\big)\exp\Big(C\int_0^t\big(1 + \|\Lambda_{h}^{\frac{\alpha}{2}}\omega\|_{L^2}^{2}\big)\dd t\Big)\\
 & \leq\big(\|\partial_{z}\omega_0\|^2_{L^2}+1\big)\exp\Big(C\big(T+\|\omega_0\|_{L^2}^2\big)\Big).
\end{align*}
The desired estimate (\ref{eq:omegazimprove}) follows immediately.
\end{proof}

\vskip .1in
With (\ref{eq:omegazimprove}) in hand, we establish the horizontal derivative estimation of $\omega$.
\begin{lemma}\label{lem:omegaimprove}
Let $T>0$ and $\alpha\in [\alpha_{0},\frac{6}{5})$. Then we have
\begin{equation}\label{eq:omegaimprove}
\|\Lambda_{h}^{\rho}\omega(t)\|_{L^{2}}^{2}
+\nu\int_{0}^{t}{\|\Lambda_{h}^{\rho+\frac{\alpha}{2}}\omega(\tau)\|_{L^{2}} ^{2} \,d\tau}\leq
C, \quad\forall\, t\in[0,T],
\end{equation}
for any $\rho\in[0,\rho_1]$, where
\begin{equation}\label{eq:rho1}
\rho_1 = \min\Big\{\frac{2\alpha^2-\alpha}{8-4\alpha},\,\frac{2\alpha^2+\alpha-2}{4}\Big\},	
\end{equation}
and $C$ depends on $\|\Lambda_{h}^{\delta_1}u_0\|_{L^2}$, $\|\Lambda_{h}^{\delta_*}u_0\|_{L^2}$, $\|\Lambda_{h}^{\rho}\omega_0\|_{L^{2}}$, $\|\partial_z\omega_0\|_{L^{2}}$, $\|\omega_0\|_{L^\infty}$ and $T$.
\end{lemma}
\begin{proof}
Taking the $L^2$ inner product of \eqref{eq:omega} with  $\Lambda_{h}^{2\rho}\omega$, it implies
\begin{align}\label{eq:omegarho}
&\frac{1}{2}\frac{\dd}{\dd t}\|\Lambda_{h}^{\rho}\omega(t)\|_{L^{2}}^{2}
+\nu\|\Lambda_{h}^{\rho+\frac{\alpha}{2}}\omega\|_{L^{2}}^{2} =
-\int_{\Omega}\big(u\partial_{x}\omega + \widetilde{w}\partial_{z}\omega\big)\Lambda_{h}^{2\rho}\omega \dxdz \nonumber\\ 
& = -\int_{\Omega}\Big(\partial_{x}(u \omega) + \partial_{z}(\widetilde{w}\omega)\Big) \Lambda_{h}^{2\rho}\omega\, \dxdz
\triangleq N_{1}+N_{2}.
\end{align}
Note that in the second equality, we have used the incompressibility condition \eqref{eq:main}$_2$ so
\[
 \partial_{x}(u \omega) + \partial_{z}(\widetilde{w}\omega) 
 = u\partial_{x}\omega + \widetilde{w}\partial_{z}\omega + (\partial_x u + \partial_z\widetilde{w})\omega
 = u\partial_{x}\omega + \widetilde{w}\partial_{z}\omega.
\]

For $N_1$, using fractional Leibniz rule, \eqref{eq:omegaMP}, \eqref{eq:uLinf}, interpolation \eqref{eq:interp1}, and Young's inequality, the term $N_{12}$ can be estimated as
 \begin{align}\label{eq:N1}
N_1 & = \int_{\Omega}\Lambda_{h}^{\rho-\frac{\alpha}{2}}\partial_{x}(u\omega)\, \Lambda_{h}^{\rho+\frac{\alpha}{2}}\omega\, \dxdz
\leq C\|\Lambda_{h}^{\rho-\frac{\alpha}{2}+1}(u\omega)\|_{L^{2}} \|\Lambda_{h}^{\rho+\frac{\alpha}{2}}\omega\|_{L^{2}} \nonumber\\
&\leq C\big(\|\omega\|_{L^{\infty}} \|\Lambda_{h}^{\rho-\frac{\alpha}{2}+1}u\|_{L^{2}} + \|u\|_{L^{\infty}} \|\Lambda_{h}^{\rho-\frac{\alpha}{2}+1}\omega\|_{L^{2}}\big) \|\Lambda_{h}^{\rho+\frac{\alpha}{2}} \omega\|_{L^{2}} \nonumber\\
&\leq C\|\omega_0\|_{L^{\infty}} \big( \|\Lambda_{h}^{\alpha}u\|_{L^{2}}+ \|\Lambda_{h}^{\rho}\omega\|_{L^{2}}^{\frac{2\alpha-2}{\alpha}} \|\Lambda_{h}^{\rho+\frac\alpha2}\omega\|_{L^{2}}^{\frac{2-\alpha}\alpha}\big) \|\Lambda_{h}^{\rho+\frac{\alpha}{2}} \omega\|_{L^{2}} \nonumber\\
&\leq \frac{\nu}{4}\|\Lambda_{h}^{\rho+\frac{\alpha}{2}} \omega\|_{L^{2}}^2 + C\big(\|\Lambda_{h}^{\alpha} u \|_{L^{2}}^{2} + \|\Lambda_{h}^{\rho}\omega\|_{L^{2}}^{2}\big),
\end{align}
where we absorb the $\|\omega_0\|_{L^{\infty}}$ dependence to the constant $C$. Note that in the penultimate inequality, we have applied Poincar\'e inequality to obtain $\|\Lambda_{h}^{\rho+1-\frac{\alpha}{2}}u\|_{L^{2}}\leq C\|\Lambda_{h}^{\alpha}u\|_{L^{2}}$, which only requires a mild assumption $
	\rho+1-\frac\alpha2\leq\alpha$, which is valid as long as \eqref{eq:rhocond1} holds. 

The treatment of the term $N_{2}$ is more delicate, relying essentially on the enhanced energy bound for $\|\Lambda_{h}^{\delta} u\|_{L^2}$ and the a priori estimate \eqref{eq:omegaMP} on $\|\omega\|_{L^\infty}$. 
We adopt the analysis analogous to \eqref{eq:U2} (replacing $\delta$ by $2\rho$). It yields
\begin{align*}
N_{2} & = \int_{\Omega}\widetilde{w}\,\omega\, \Lambda_{h}^{2\rho}\partial_{z}\omega\, \dxdz
= \int_{\Omega}\Lambda_{h}^{2\rho-\frac{\alpha}{2}}(\widetilde{w}\omega )\, \Lambda_{h}^{\frac{\alpha}{2}}\partial_{z}\omega\, \dxdz
\leq C\|\Lambda_{h}^{2\rho-\frac{\alpha}{2}}(\widetilde{w}\omega ) \|_{L^{2}}\|\Lambda_{h}^{\frac{\alpha}{2}}\partial_{z}\omega\|_{L^{2}}\\
& \leq C \Big(\|\omega_0\|_{L^{\infty}}\|\Lambda_{h}^{2\rho-\frac{\alpha}{2}} \partial_xu\|_{L^2} + \|\partial_xu\|_{L_{z}^{2}L_{x}^{\frac{2p}{p-2}}} \|\Lambda_{h}^{2\rho-\frac{\alpha}{2}}\omega\|_{L_{z}^{2}L_{x}^{p}}\Big)\|\Lambda_{h}^{\frac{\alpha}{2}}\partial_{z}\omega\|_{L^{2}}
 \triangleq N_{21} + N_{22},
\end{align*}
where the parameter $p\in(2,\infty)$ will be chosen later in \eqref{eq:p2}.

The term $N_{21}$ can be estimated by:
\begin{equation}\label{eq:N21}
N_{21} \leq C \|\omega_0\|_{L^\infty} \|\Lambda_{h}^{2\rho-\frac\alpha2+1} u\|_{L^2}\|\Lambda_{h}^{\frac{\alpha}{2}}\partial_{z}\omega\|_{L^{2}}
 \leq C\|\omega_0\|_{L^{\infty}}\big(\|\Lambda_{h}^{\delta+\frac\alpha2} u\|_{L^2}^2+\|\Lambda_{h}^{\frac{\alpha}{2}}\partial_{z}\omega\|_{L^{2}}^2\big),	
\end{equation}
provided that
\begin{equation}\label{eq:rhocond1}
	2\rho-\frac\alpha2+1\leq \delta+\frac\alpha2,\quad\Longleftrightarrow\quad \rho\leq \frac{\delta+\alpha-1}{2}.
\end{equation}

For $N_{22}$, $\|\partial_xu\|_{L_{z}^{2}L_{x}^{\frac{2p}{p-2}}}$ can be estimated by \eqref{eq:ux}, namely
\[
\|\partial_{x}u\|_{L_{z}^{2}L_{x}^{\frac{2p}{p-2}}} \leq  C\|\Lambda_{h}^{1+\frac{1}{p}}u\|_{L^{2}}
\leq C\|\Lambda_{h}^{\delta}u\|_{L^{2}}^{1-\theta_{21}} \|\Lambda_{h}^{\delta+\frac{\alpha}{2}}u\|_{L^{2}}^{\theta_{21}},\quad \text{where}\quad \theta_{21} = \tfrac{2-2\delta+\frac{2}{p}}{\alpha},
\]
with a different $p$ defined in \eqref{eq:p2}.
 We use \eqref{eq:interp2} and interpolate $\|\Lambda_{h}^{2\rho-\frac{\alpha}{2}}\omega\|_{L_{x}^{p}}$ by $\|\omega\|_{L^\infty_x}$ and $\|\Lambda_{h}^{\rho+\frac{\alpha}{2}}\omega\|_{L_{x}^{2}}$. Note that this offers further enhancement compared with \eqref{eq:omegaenhance}. We choose
\begin{equation}\label{eq:p2}
 p \triangleq \frac{4\rho+2\alpha}{4\rho-\alpha}\in(2,\infty), \quad \text{so that}\quad
 \theta_{22} = \frac2p = \frac{2\rho-\frac\alpha2}{\rho+\frac\alpha2},
\end{equation}
and obtain
\begin{equation}\label{eq:omegaenhance2}
 \|\Lambda_{h}^{2\rho-\frac{\alpha}{2}}\omega\|_{L^p_x}\leq C \|\omega\|_{L^\infty_x}^{1-\theta_{22}}\|\Lambda_{h}^{\rho+\frac{\alpha}{2}}\omega\|_{L^2_x}^{\theta_{22}} 
  = C \|\omega\|_{L^\infty_x}^{\frac{2\alpha-2\rho}{2\rho+\alpha}}\|\Lambda_{h}^{\rho+\frac{\alpha}{2}}\omega\|_{L^2_x}^{\frac{4\rho-\alpha}{2\rho+\alpha}}.
\end{equation}
Consequently, we apply Young's inequality and deduce
\begin{align}\label{eq:N22}
N_{22} &\leq C \|\omega_{0}\|_{L^{\infty}}^{1-\theta_{22}}\|\Lambda_{h}^{\delta}u\|_{L^{2}}^{1-\theta_{21}} \|\Lambda_{h}^{\delta+\frac{\alpha}{2}}u\|_{L^{2}}^{\theta_{21}} \|\Lambda_{h}^{\frac{\alpha}{2}}\partial_{z}\omega\|_{L^{2}} \|\Lambda_{h}^{\rho+\frac{\alpha}{2}}\omega\|_{L^2_x}^{\theta_{22}} \nonumber\\
&\leq \frac{\nu}{4}\|\Lambda_{h}^{\rho+\frac{\alpha}{2}}\omega\|_{L^{2}}^{2}
+C(\|\Lambda_{h}^{\delta+\frac{\alpha}{2}} u \|_{L^{2}}^{\frac{2\theta_{21}}{1-\theta_{22}}}+\|\Lambda_{h}^{\frac{\alpha}{2}}\partial_{z}\omega\|_{L^{2}}^{2}),	
\end{align}
where we absorb $\|\omega_{0}\|_{L^{\infty}}$ and $\|\Lambda_{h}^{\delta}u\|_{L^{2}}$ (thanks to \eqref{eq:uimprove}) to the constant $C$. Since $\|\Lambda_{h}^{\delta+\frac\alpha2}u\|_{L^{2}}\in L^2(0,T)$, we require
\begin{equation}\label{eq:rhocond}
\frac{2\theta_{21}}{1-\theta_{22}} = \frac{(8-4\delta)\rho+(1-2\delta)\alpha}{\alpha(\alpha-\rho)}\leq 2,\quad \Longleftrightarrow\quad
\rho\leq \frac{\alpha(2\alpha+2\delta-1)}{2(\alpha-2\delta+4)}.
\end{equation}
From the two conditions \eqref{eq:rhocond1} and \eqref{eq:rhocond}, we define
\begin{equation}\label{eq:g}
 g(\delta) \triangleq \min\Big\{\frac{\alpha(2\alpha+2\delta-1)}{2(\alpha-2\delta+4)},\,\frac{\delta+\alpha-1}{2}\Big\} = 
 \begin{cases}
 	\frac{\alpha(2\alpha+2\delta-1)}{2(\alpha-2\delta+4)}, &\delta\in[\frac{2-\alpha}{2}, 2-\alpha],\\
 	\frac{\delta+\alpha-1}{2},&\text{otherwise.} 
 \end{cases}
\end{equation}
Taking $\delta = \delta_1$ in \eqref{eq:delta1}, the conditions \eqref{eq:rhocond1} and \eqref{eq:rhocond} hold when $\rho\leq g(\delta_1)=\rho_1$, where $\rho_1$ is defined in \eqref{eq:rho1}.

Inserting the estimates \eqref{eq:N1}, \eqref{eq:N21} and \eqref{eq:N22} into \eqref{eq:omegarho}, and plugging the value of $\delta$ and $\rho$ in \eqref{eq:delta1} and \eqref{eq:rho1}, respectively, it follows that
\[
 \frac{\dd}{\dd t}\|\Lambda_{h}^{\rho}\omega(t)\|_{L^{2}}^{2} + \nu\|\Lambda_{h}^{\rho+\frac{\alpha}{2}}\omega\|_{L^{2}}^{2}
 \leq C \Big(\|\Lambda_{h}^{\alpha} u \|_{L^{2}}^{2} + \|\Lambda_{h}^{\delta+\frac{\alpha}{2}} u \|_{L^{2}}^{2}+\|\Lambda_{h}^{\frac{\alpha}{2}}\partial_{z}\omega\|_{L^{2}}^{2}\Big) + C\|\Lambda_{h}^{\rho}\omega\|_{L^{2}}^{2}.
\]
From the a priori bounds \eqref{eq:uE1}, \eqref{eq:uimprove} and \eqref{eq:omegazimprove}, we know 
\[\|\Lambda_{h}^{\alpha} u \|_{L^{2}}^{2} + \|\Lambda_{h}^{\delta+\frac{\alpha}{2}} u \|_{L^{2}}^{2}+\|\Lambda_{h}^{\frac{\alpha}{2}}\partial_{z}\omega\|_{L^{2}}^{2}\in L^1(0,T).\]
Applying Gr\"onwall inequality, we obtain
\begin{align*}
 \|\Lambda_{h}^{\rho}\omega(t)\|^2_{L^2} + \nu\int_0^t\|\Lambda_{h}^{\rho+\frac\alpha2}\omega(\tau)\|_{L^2}^2 \dd \tau
 &\leq C\|\Lambda_{h}^{\rho}\omega_0\|^2_{L^2}\,e^C, 
\end{align*}
where $C$ depends on $\|\Lambda_{h}^{\delta}u_0\|_{L^2}$, $\|\partial_z\omega_0\|_{L^{2}}$, $\|\omega_0\|_{L^\infty}$ and $T$.
This completes the proof of \eqref{eq:omegaimprove}.
\end{proof}

\section{Proof of Theorem \ref{thm:GWP}}\label{sec:GWP}
In this section, we establish the global regularity theory for \eqref{eq:main}, with $\alpha\in[\alpha_0,\frac65)$, making effective use of the enhanced a priori bounds obtained in Section \ref{sec:apriori}.

The goal is to verify the regularity criterion \eqref{eq:BKM}, that is,
\[    \int_0^T\|\Lambda_{h}^{\frac{3-\alpha}{2}}\omega(t)\|_{L^2}^2\,\dd t<\infty.
\]
Then, global well-posedness follows from the theory developed in \cite{abdo2025well}.

From Lemma \ref{lem:omegaimprove}, we have the enhanced a priori bound
\[
  \int_0^T\|\Lambda_{h}^{\rho + \frac\alpha2}\omega(t)\|_{L^2}^2 \,\dd t<\infty,\quad\forall\,\rho\in[0,\rho_1].
\]
Hence, the regularity criterion \eqref{eq:BKM} is satisfied whenever 
\[
 \rho_1\geq\rho^*\triangleq\frac{3-2\alpha}{2},\quad\Longleftrightarrow\quad
 \alpha \geq \alpha_1 \triangleq \frac{13-\sqrt{73}}{4} \approx 1.1140.
\]
Therefore, we obtain global well-posedness for $\alpha\in[\alpha_1,\frac65)$.
 
However, since $\alpha_0 < \alpha_1$, a more refined estimate is needed to cover the range $\alpha \in [\alpha_0, \alpha_1)$. This refinement can be achieved by iteratively applying the enhancement procedure developed in Lemmas \ref{lem:uimprove} and \ref{lem:omegaimprove}. 

To begin with, we modify Lemma \ref{lem:uimprove}, making use of the a priori bound on $\|\Lambda_{h}^{\rho_1}\omega\|_{L^2}$ to obtain a further improvement in the a priori control of $\|\Lambda_{h}^{\delta}u\|_{L^2}$ for larger values of $\delta$.

\begin{lemma}\label{lem:uimprove2}
Let $T>0$ and $\alpha\in[\alpha_0,\alpha_1)$. Then we have the bound \eqref{eq:uimprove}, namely
\[
\|\Lambda_{h}^{\delta}u(t)\|^2_{L^2} + \nu\int_0^t\|\Lambda_{h}^{\delta+\frac{\alpha}{2}}u(\tau)\|_{L^2}^2  \dd\tau\leq C,\quad\forall\, t\in[0,T],
\]
for any $\delta\in (\delta_1,\delta_2]$, where
\begin{align}\label{eq:delta2}
\delta_2 \triangleq \frac{\alpha(2\alpha-1)(2-\alpha)}{2(4-\alpha-2\alpha^2)},
\end{align}
and $C$ depends on $\|\Lambda_{h}^{\delta}u_0\|_{L^2}$, $\|\Lambda_{h}^{\delta_*}u_0\|_{L^2}$, $\|\Lambda_{h}^{\rho_1}\omega_0\|_{L^{2}}$, $\|\partial_z\omega_0\|_{L^{2}}$, $\|\omega_0\|_{L^\infty}$ and $T$.
\end{lemma}
\begin{proof}
 The proof follows from Lemma \ref{lem:uimprove}. We improve the bound \eqref{eq:omegaenhance}, replacing $\|\Lambda_{h}^{\frac{\alpha}{2}}\omega\|_{L^2_x}$ with  $\|\Lambda_{h}^{\rho+\frac{\alpha}{2}}\omega\|_{L^2_x}$ in the interpolation:
 \begin{equation}\label{eq:omegaenhance3}
  \|\Lambda_{h}^{\delta-\frac{\alpha}{2}}\omega\|_{L^p_x}\leq C \|\omega\|_{L^\infty_x}^{1-\theta_{12}}\|\Lambda_{h}^{\rho+\frac{\alpha}{2}}\omega\|_{L^2_x}^{\theta_{12}} 
  = C \|\omega\|_{L^\infty_x}^{\frac{2\alpha+2\rho-2\delta}{2\rho+\alpha}}\|\Lambda_{h}^{\rho+\frac{\alpha}{2}}\omega\|_{L^2_x}^{\frac{2\delta-\alpha}{2\rho+\alpha}},
 \end{equation}
 where $p$ and $\theta_{12}$ are redefined by
 \[
  p \triangleq \frac{2(2\rho+\alpha)}{2\delta-\alpha}\in (2,\infty),\quad \text{so that} \quad
  \theta_{12}=\frac{2}{p} = \frac{\delta-\frac\alpha2}{\rho+\frac\alpha2},
 \]
 which is analogous to \eqref{eq:omegaenhance2}. Together with \eqref{eq:ux}, the estimate on $U_{22}$ becomes
 \begin{align}\label{eq:U22e}
 U_{22} & \leq C \|\omega_0\|_{L^\infty}^{1-\theta_{12}}\|\Lambda_{h}^{\rho+\frac{\alpha}{2}}\omega\|_{L^2}^{\theta_{12}}\|\Lambda_{h}^{\delta}u\|_{L^{2}}^{1-\theta_{11}} \|\Lambda_{h}^{\delta+\frac{\alpha}{2}}u\|_{L^{2}}^{1+\theta_{11}} \nonumber\\
 &\leq \frac{\nu}{8} \|\Lambda_{h}^{\delta+\frac{\alpha}{2}}u\|_{L^2}^2 + C\|\omega_0\|_{L^{\infty}}^{\frac{2(1-\theta_{12})}{1-\theta_{11}}}\|\Lambda_{h}^{\rho+\frac{\alpha}{2}}\omega\|_{L^{2}}^{\frac{2\theta_{12}}{1-\theta_{11}}}\|\Lambda_{h}^{\delta}u\|_{L^{2}}^{2}.
\end{align}
To apply Gr\"onwall inequality, we require
\begin{equation}\label{eq:deltacond}
 \frac{2\theta_{12}}{1-\theta_{11}}=\frac{2\alpha(2\delta-\alpha)}{(\alpha-1)(2\delta+\alpha)+2\rho(\alpha+2\delta-2)}\leq2,\quad\Longleftrightarrow\quad
 \delta\leq\frac{2\alpha^2-\alpha + 2\rho(\alpha-2)}{2(1-2\rho)}\triangleq f(\rho).	
\end{equation}
Note that if we take $\rho=0$, \eqref{eq:deltacond} reduces to the bound \eqref{eq:deltabound} in Lemma \ref{lem:uimprove} with $\delta_1 = f(0)$. 
Thanks to the a priori bound on $\|\Lambda_{h}^{\rho_1}\omega\|_{L^2}$, we plug in $\rho=\rho_1$ and deduce that 
\[\delta\leq f(\rho_1) = \delta_2,\]
where we note that $\rho_1 = \frac{2\alpha^2-\alpha}{8-4\alpha}$ when $\alpha\in[\alpha_0,\alpha_1)$.
The rest of the argument is analogous to Lemma \ref{lem:uimprove}, and we conclude with the bound \eqref{eq:uimprove}.
\end{proof}

Having a priori bound on $\|\Lambda_{h}^{\delta_2}u\|_{L^2}$, we reapply Lemma \ref{lem:omegaimprove}, and take $\delta=\delta_2$ in \eqref{eq:delta2}. This leads to the bound \eqref{eq:omegaimprove} for $\rho\in(\rho_1,\rho_2]$, where
 \[\rho_2\triangleq g(\delta_2) = \min\Big\{\frac{\alpha(2\alpha+2\delta_2-1)}{2(\alpha-2\delta_2+4)},\,\frac{\delta_2+\alpha-1}{2}\Big\}.\]
The regularity criterion \eqref{eq:BKM} is satisfied whenever
\[
 \rho_2\geq\rho^*,\quad\Longleftrightarrow\quad
 \alpha \geq \alpha_2 \approx 1.0635,
\]
where $\alpha_2$ is the root of the cubic equation $6\alpha^3+17\alpha^2-70\alpha+48=0$.
Importantly, $\alpha_2<\alpha_0$. Therefore, we obtain global well-posedness over the entire range $[\alpha_0, \tfrac{6}{5})$.
See the right panel of Figure \ref{fig1} for illustration.

\begin{remark}
 The iterations can continue, offering a priori bounds for higher derivatives of $u$ and $\omega$. Indeed, from conditions \eqref{eq:deltacond}	and \eqref{eq:rhocond}, we have the following iterative scheme:
 \[
 \delta_{k+1} = f(\rho_k) = \frac{2\alpha^2-\alpha + 2\rho_k(\alpha-2)}{2(1-2\rho_k)},\quad
 \rho_{k+1} = g(\delta_{k+1}),
 \]
 with $\rho_0=0$, and $g$ is defined in \eqref{eq:g}. Note that 
 \[
 g^{-1}(\rho) = \begin{cases}
 	\frac{\alpha+4}{2} - \frac{3\alpha(\alpha+1)}{2(\alpha+2\rho)}, & \rho\in[\frac\alpha4,\frac12],\\
 	2\rho-\alpha+1, & \text{otherwise}.
 \end{cases}
 \]
Since $\{\rho_k\}$ is increasing with $\rho_1\geq\frac\alpha4$ and $\rho^*<\frac12$, we have the following recursive formula:
\[
 \rho_{k+1} = g(\delta_{k+1}) = \frac{\alpha(2\alpha+2\delta_{k+1}-1)}{2(\alpha-2\delta_{k+1}+4)} = \frac{\alpha(2\alpha-1-2\rho_k)}{4(2-\alpha-2\rho_k)},\quad \rho_0=0,
\] 
as long as $\rho_{k+1}\leq\rho^*$. Analysis of this formula reveals a dichotomy:
 \begin{itemize}
  \item If $\alpha\in(1,\frac{4}{\sqrt{15}}]$, then $\rho_k$ has an upper bound: 
  \[\lim_{k\to\infty}\rho_k = \frac{4-\alpha-\sqrt{16-15\alpha^2}}{8}=\rho_M<\rho^*.\]
 Hence, the criterion \eqref{eq:BKM} cannot be verified. See the left panel of Figure \ref{fig1} for illustration.
  \item If $\alpha\in(\frac{4}{\sqrt{15}},\frac65)$, $\rho_k$ exceeds $\frac12$ after finitely many steps, allowing the enhancement procedure to reach the desired $\rho^*$. See the right panel of Figure \ref{fig1} for illustration.
 \end{itemize}
This shows that, even if the crucial assumption $\alpha \geq \alpha_0$ were removed, our approach still does not extend directly to the case $\alpha>1$.
\end{remark}

\begin{figure}[ht]
\includegraphics[width=.45\linewidth]{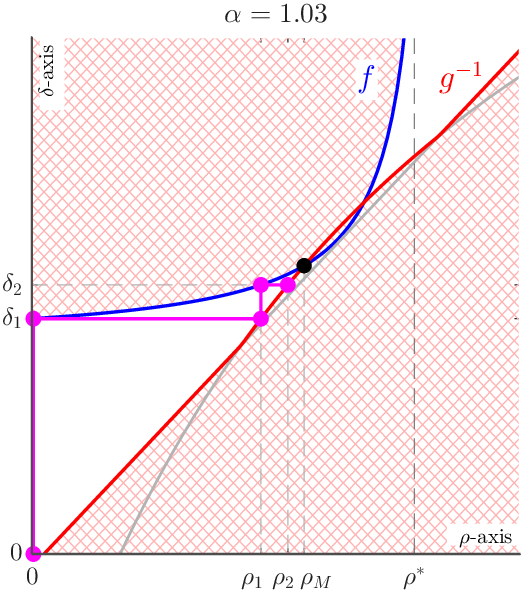}\quad	
\includegraphics[width=.45\linewidth]{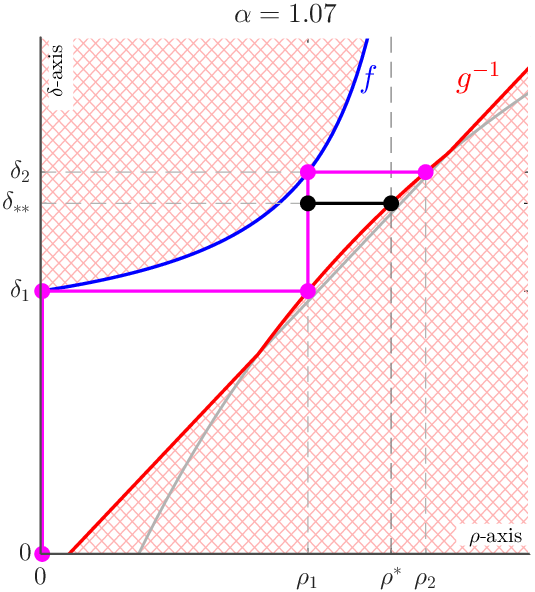}
\caption{An illustration of the iteration scheme. Left ($1<\alpha\leq\frac4{\sqrt{15}}$): the maximal enhancement on $\rho$ is $\rho_M$, which is smaller than $\rho^*$. Right ($\frac4{\sqrt{15}}<\alpha<\frac65$): there is a path such that $\rho$ reaches $\rho^*$, and the minimal requirement on $\delta$ is $\delta_{**}$.}\label{fig1}
\end{figure}

Finally, we determine the minimal regularity requirement on $u_0$ that ensures global regularity. In the last iteration $k+1$, such that $\rho_k \leq \rho^* < \rho_{k+1}$, we apply Lemma \ref{lem:omegaimprove} with $\rho = \rho^*$. In this case, the condition \eqref{eq:rhocond} becomes
\[
\rho^*\leq \frac{\alpha(2\alpha+2\delta-1)}{2(\alpha-2\delta+4)},\quad \Longleftrightarrow\quad
\delta\geq \frac{2(-\alpha^2-\alpha+3)}{3-\alpha}\triangleq \delta_{**}.
\]
Hence, we may replace $\delta_{k+1}$ by the (smaller) quantity $\delta_{**}$ in Lemma \ref{lem:omegaimprove} (see the right panel of Figure \ref{fig1}). Consequently, the regularity criterion \eqref{eq:BKM} holds if the initial data satisfy
\[
\Lambda_{h}^{\delta_*}u_0\in L^2,\quad \Lambda_{h}^{\delta_{**}}u_0\in L^2,\quad \Lambda_{h}^{\frac{3-2\alpha}{2}}\omega_0\in L^2,\quad  \partial_z\omega_0\in L^2,\quad \text{and}\quad  \omega_0\in L^\infty.
\]
Define 
\[
	\delta_m \triangleq \max \{\delta_*,\, \delta_{**}\} = \max \Big\{\frac{\alpha}{2},\,\frac{-2\alpha^2+2\alpha+1}{\alpha},\, \frac{2(-\alpha^2-\alpha+3)}{3-\alpha}\Big\}.
\]
Then, $\delta_m$ is the minimum initial regularity on $u_0$ that guarantees global well-posedness of the system \eqref{eq:main}. Observe that 
\[
 \frac{-2\alpha^2+2\alpha+1}{\alpha}\leq \max \Big\{\frac{\alpha}{2},\, \frac{2(-\alpha^2-\alpha+3)}{3-\alpha}\Big\},\quad \forall\,\alpha\in[\alpha_0,\tfrac65).
\]
Therefore, $\delta_m$ can be expressed as in \eqref{eq:deltam}. The proof of Theorem \ref{thm:GWP} is thus complete.

\section{Proof of Theorem \ref{thm:small}}\label{sec:GWPsmall}
In this section, we focus on the case $\alpha \in [1, \alpha_0)$ and establish the global well-posedness of the system \eqref{eq:main} with small initial data.

In \cite[Theorem 4.9]{abdo2025well}, it was shown that for $\alpha \in (1,2)$, the system \eqref{eq:main} admits a unique global solution provided the initial data are sufficiently small. More precisely, there exists a constant $c > 0$ such that global well-posedness holds whenever
\[
 \|\omega_0\|_{L^\infty}<c,\quad \text{and}\quad \|\Lambda_{h}^{\alpha} u_0\|^2_{L^2}+\|\Lambda_{h}^{\frac{\alpha}{2}}\omega_0\|_{L^{2}}^{2}+ \|\partial_{z}\omega_0\|^2_{L^2}<c.
\]
The result can be extended to the critical regime $(\alpha=1)$, see \cite[Theorem 5.9]{abdo2025well}.

Our goal here is to remove the second smallness assumption, so that smallness is required only for $\|\omega_0\|_{L^\infty}$. In addition, our theory imposes weaker regularity requirements on the initial data.

We propagate the three quantities $\|\Lambda_{h}^{\delta}u\|_{L^2}$, $ \|\Lambda_{h}^{\rho}\omega\|_{L^2}$ and $\|\partial_z\omega\|_{L^2}$ simultaneously by applying Lemmas \ref{lem:uimprove}, \ref{lem:omegaimprove}, and \ref{lem:omegazimprove}, respectively. We will choose $\delta$ and $\rho$ in \eqref{eq:rhodelta}. Define
\begin{align*}
X(t) & = \|\Lambda_{h}^{\delta}u(t)\|^2_{L^2} + \|\Lambda_{h}^{\rho}\omega(t)\|^2_{L^2} +\|\partial_z\omega(t)\|^2_{L^2},\\
Y(t) & = \|\Lambda_{h}^{\delta+\frac{\alpha}{2}}u(t)\|_{L^2}^2 + \|\Lambda_{h}^{\rho+\frac{\alpha}{2}}\omega(t)\|_{L^2}^2+\|\Lambda_{h}^{\frac\alpha2}\partial_z\omega(t)\|^2_{L^2}.
\end{align*}
Then we have
\begin{equation}\label{eq:coupled}
 \frac12\frac{\dd}{\dd t} X(t) +\nu Y(t) = U_1 + U_{21} + U_{22} + N_1 + N_{21} + N_{22} + J_1 + J_2 \triangleq \sum_{i=1}^8 R_i(t).
\end{equation}

We now state the following key lemma.
\begin{lemma}\label{lem:small}
 Let $\alpha\in[1,\alpha_0)$, and
 \begin{equation}\label{eq:rhodelta}
\rho = \rho^* = \frac{3-2\alpha}{2},\quad\text{and}\quad \delta= \delta_{**} = \frac{2(-\alpha^2-\alpha+3)}{3-\alpha}. 	
 \end{equation}
 Then, for any $t>0$, each term on the right-hand side of \eqref{eq:coupled} can be bounded by
\begin{equation}\label{eq:smallICbound}
  R_i(t)\leq C_i\|\omega_0\|_{L^\infty}^{\mu_i} X(t)^{\frac{1-{\mu_i}}2}Y(t),\quad i=1,\ldots,8,
\end{equation}
where the parameter $\mu_i\in(0,1]$, and $C_i$ is a universal constant independent of initial data.
\end{lemma}

Given Lemma \ref{lem:small}, we are ready to prove Theorem \ref{thm:small}.
Set
\begin{equation}\label{eq:c}
 c \triangleq \min_{i=1,\ldots,8}\Big(\frac{\nu}{16C_i}\Big)^{\frac{1}{\mu_i}}X_0^{-\frac{1-\mu_i}{2\mu_i}}>0.
\end{equation}
By this choice we have
\[
C_i\|\omega_0\|_{L^\infty}^{\mu_i} X_0^{\frac{1-{\mu_i}}2}\leq \frac{\nu}{16},\quad\forall \, i=1,\ldots,8.
\]
Suppose $X(t_*) = X_0$ at time $t_*$. Then, from \eqref{eq:coupled} and \eqref{eq:smallICbound} we obtain
\[
\frac12\frac{\dd}{\dd t}X(t_*)=\sum_{i=1}^8\Big(-\frac\nu8 Y(t_*)+R_i(t_*)\Big)
\leq -\sum_{i=1}^8\Big(\frac\nu8 - C_i\|\omega_0\|_{L^\infty}^{\mu_i} X_0^{\frac{1-{\mu_i}}2}\Big)Y(t_*)\leq0.
\]
Therefore, $X(t)$ cannot exceed $X_0$, that is,
\[
X(t)\leq X_0,\quad\forall\,t\geq0.
\]
Consequently, for every $t\geq0$ we have
\[
\frac{\dd}{\dd t}X(t) + \nu Y(t)=2\sum_{i=1}^8\Big(-\frac{\nu}{16} Y(t)+R_i(t)\Big)
\leq -2\sum_{i=1}^8\Big(\frac{\nu}{16} - C_i\|\omega_0\|_{L^\infty}^{\mu_i} X_0^{\frac{1-{\mu_i}}2}\Big)Y(t)\leq0,
\]
which implies
\[
\int_0^T Y(t)\,\dd t\leq \frac1\nu X_0<\infty,\quad\text{i.e.}\quad
Y\in L^1(0,T).
\]
In particular, the regularity criterion \eqref{eq:BKM} is satisfied:
\[\int_0^T\|\Lambda_{h}^{\frac{3-\alpha}{2}}\omega(t)\|_{L^2}^2\,\dd t = \int_0^T\|\Lambda_{h}^{\rho^*+\frac{\alpha}{2}}\omega(t)\|_{L^2}^2\,\dd t<\infty.\]
Therefore the solution extends globally, completing the proof of Theorem \ref{thm:small}. 
\vskip 1em

It remains to prove Lemma \ref{lem:small}. We treat the subcritical regime $\alpha>1$ and the critical regime $\alpha=1$ separately.

\subsection{Subcritical regime}
Assume $\alpha\in(1,\alpha_0)$. We now verify the bound \eqref{eq:smallICbound} for each term $R_i$, $i=1,\ldots,8$.

For $U_1$ and $U_{21}$, from \eqref{eq:U1}, \eqref{eq:U21} and Poincar\'e inequality, we have
\begin{align*}
 U_1 + U_{21} & \leq C\|\omega_0\|_{L^{\infty}}\|\Lambda_{h}^{\delta}u\|_{L^2}^{\frac{2\alpha-2}{\alpha}} \|\Lambda_{h}^{\delta+\frac{\alpha}{2}}u\|_{L^2}^{\frac{2}{\alpha}}\leq C\|\omega_0\|_{L^{\infty}} \|\Lambda_{h}^{\delta+\frac{\alpha}{2}}u\|_{L^2}^2\leq 	C\|\omega_0\|_{L^{\infty}} Y(t).
\end{align*}
Then \eqref{eq:smallICbound} holds with $\mu_1=\mu_2=1.$
For $U_{22}$, we deduce from \eqref{eq:U22e} that
\begin{align*}
U_{22} & \leq C \|\omega_0\|_{L^\infty}^{1-\theta_{12}}\|\Lambda_{h}^{\rho+\frac{\alpha}{2}}\omega\|_{L^2}^{\theta_{12}}\|\Lambda_{h}^{\delta}u\|_{L^{2}}^{1-\theta_{11}} \|\Lambda_{h}^{\delta+\frac{\alpha}{2}}u\|_{L^{2}}^{1+\theta_{11}}\leq C \|\omega_0\|_{L^\infty}^{1-\theta_{12}} \|\Lambda_{h}^{\delta}u\|_{L^{2}}^{1-\theta_{11}} Y(t)^{\frac{1+\theta_{11}+\theta_{12}}{2}}.
\end{align*}
To obtain \eqref{eq:smallICbound}, we require $\frac{1+\theta_{11}+\theta_{12}}{2}\leq1$, which is precisely \eqref{eq:deltacond}: $\delta\leq f(\rho)$. Under the assumption, we use Poincar\'e inequality and deduce
\[
 U_{22} \leq C\|\omega_0\|_{L^\infty}^{1-\theta_{12}}\|\Lambda_{h}^\delta u\|_{L^{2}}^{\theta_{12}}\, Y(t)\leq C\|\omega_0\|_{L^\infty}^{1-\theta_{12}} X(t)^{\frac{\theta_{12}}{2}}Y(t),
\]
which has the form \eqref{eq:smallICbound} with 
\[\mu_3=1-\theta_{12}=\frac{2\alpha+2\rho-2\delta}{2\rho+\alpha}\in(0,1]
\quad\Longleftrightarrow\quad
\frac\alpha2\leq\delta<\alpha+\rho.\]

For $N_1$, from \eqref{eq:N1} and Poincar\'e inequality, we have
\[
N_1 \leq C\|\omega_0\|_{L^{\infty}} \big( \|\Lambda_{h}^{\alpha}u\|_{L^{2}}+ \|\Lambda_{h}^{\rho}\omega\|_{L^{2}}^{\frac{2\alpha-2}{\alpha}} \|\Lambda_{h}^{\rho+\frac\alpha2}\omega\|_{L^{2}}^{\frac{2-\alpha}\alpha}\big) \|\Lambda_{h}^{\rho+\frac{\alpha}{2}} \omega\|_{L^{2}}\leq C\|\omega_0\|_{L^\infty}Y(t).
\]
Then \eqref{eq:smallICbound} holds with $\mu_4=1.$
For $N_{21}$, we obtain from \eqref{eq:N21} that
\[
N_{21} \leq C\|\omega_0\|_{L^{\infty}}\big(\|\Lambda_{h}^{\delta+\frac\alpha2} u\|_{L^2}^2+\|\Lambda_{h}^{\frac{\alpha}{2}}\partial_{z}\omega\|_{L^{2}}^2\big) \leq C\|\omega_0\|_{L^\infty}Y(t)	,
\]
provided that \eqref{eq:rhocond1} holds. Then \eqref{eq:smallICbound} holds with $\mu_5=1.$
For $N_{22}$, 
\begin{align*}
N_{22} & \leq C \|\omega_{0}\|_{L^{\infty}}^{1-\theta_{22}}\|\Lambda_{h}^{\delta}u\|_{L^{2}}^{1-\theta_{21}} \|\Lambda_{h}^{\delta+\frac{\alpha}{2}}u\|_{L^{2}}^{\theta_{21}} \|\Lambda_{h}^{\frac{\alpha}{2}}\partial_{z}\omega\|_{L^{2}} \|\Lambda_{h}^{\rho+\frac{\alpha}{2}}\omega\|_{L^2_x}^{\theta_{22}}\\
& \leq C \|\omega_{0}\|_{L^{\infty}}^{1-\theta_{22}}\|\Lambda_{h}^{\delta}u\|_{L^{2}}^{1-\theta_{21}} Y(t)^{\frac{1+\theta_{21}+\theta_{22}}2}.
\end{align*}
To obtain control by \eqref{eq:smallICbound}, we require $\frac{1+\theta_{21}+\theta_{22}}{2}\leq1$, which is precisely \eqref{eq:rhocond}. Under the assumption, we use Poincar\'e inequality and deduce
\[
 N_{22} \leq C\|\omega_0\|_{L^\infty}^{1-\theta_{22}}\|\Lambda_{h}^\delta u\|_{L^{2}}^{\theta_{22}}\, Y(t)\leq C\|\omega_0\|_{L^\infty}^{1-\theta_{22}} X(t)^{\frac{\theta_{22}}{2}}Y(t),
\]
which has the form \eqref{eq:smallICbound} with 
\[\mu_6=1-\theta_{22}=\frac{2\alpha-2\rho}{2\rho+\alpha}\in(0,1]
\quad\Longleftrightarrow\quad
\frac\alpha4\leq\rho<\alpha.\]
Note that the assumptions \eqref{eq:rhocond1} and \eqref{eq:rhocond} are equivalent to $\rho\leq g(\delta)$, where $g$ is defined in \eqref{eq:g}.

Finally, for $J_1$, from \eqref{eq:J1} and Poincar\'e inequality, we have
\[
J_1\leq C\|\omega_0\|_{L^{\infty}} \|\Lambda_{h}^{\frac{\alpha}{2}}\omega\|_{L^2} \|\Lambda_{h}^{\frac{\alpha}{2}}\partial_{z}\omega\|_{L^2}\leq C\|\omega_0\|_{L^{\infty}} Y(t).
\]
Then \eqref{eq:smallICbound} holds with $\mu_7=1.$
For $J_2$, we improve the bound \eqref{eq:omegaomegaz} by \eqref{eq:omegaenhance3} (replacing $\delta$ by $1$), namely 
\[
  \|\Lambda_{h}^{1-\frac{\alpha}{2}}\omega\|_{L^p_x}\leq C \|\omega\|_{L^\infty_x}^{1-\theta_{32}}\|\Lambda_{h}^{\rho+\frac{\alpha}{2}}\omega\|_{L^2_x}^{\theta_{32}} 
  = C \|\omega\|_{L^\infty_x}^{\frac{2\alpha+2\rho-2}{2\rho+\alpha}}\|\Lambda_{h}^{\rho+\frac{\alpha}{2}}\omega\|_{L^2_x}^{\frac{2-\alpha}{2\rho+\alpha}},
 \]
 where $p$ and $\theta_{32}$ are redefined by
 \[
  p \triangleq \frac{2(2\rho+\alpha)}{2-\alpha}\in (2,\infty),\quad \text{so that} \quad
  \theta_{32}=\frac{2}{p} = \frac{1-\frac\alpha2}{\rho+\frac\alpha2}.
 \]
Consequently, the estimate \eqref{eq:J2} becomes
\begin{align*}
J_2 & \leq C \|\Lambda_{h}^{\delta}u\|_{L^2}^{\frac12} \|\omega_0\|_{L^\infty}^{\frac{1-\theta_{32}}{2}}\|\Lambda_{h}^{\rho+\frac{\alpha}{2}}\omega\|_{L^2}^{\frac{\theta_{32}}{2}} \|\partial_z\omega\|_{L^2}^{1-\theta_{33}} \|\Lambda_{h}^{\frac{\alpha}{2}}\partial_z\omega\|_{L^2}^{1+\theta_{33}}\\
& \leq \|\omega_0\|_{L^\infty}^{\frac{1-\theta_{32}}{2}} \|\Lambda_{h}^{\delta}u\|_{L^2}^{\frac12} \|\partial_z\omega\|_{L^2}^{1-\theta_{33}} Y(t)^{\frac{2+\theta_{32}+2\theta_{33}}{4}}.
\end{align*}
To obtain control by \eqref{eq:smallICbound}, we require
\begin{equation}\label{eq:h}
 \frac{2+\theta_{32}+2\theta_{33}}{4}\leq 1,\quad\Longleftrightarrow\quad
 \delta\geq\frac{\alpha(2-\alpha)}{2(2\rho+\alpha)}+\frac{-3\alpha^2+2\alpha+2}{2\alpha} \triangleq h(\rho).
\end{equation}
Note that when $\rho=0$, \eqref{eq:h} reduces to \eqref{eq:deltalow}. Under the assumption \eqref{eq:h}, we apply Poincar\'e inequality and get
\[
J_2\leq \|\omega_0\|_{L^\infty}^{\frac{1-\theta_{32}}{2}} \|\Lambda_{h}^{\delta}u\|_{L^2}^{-\frac{1-\theta_{32}}{2}+\theta_{33}} \|\partial_z\omega\|_{L^2}^{1-\theta_{33}} Y(t) \leq \|\omega_0\|_{L^\infty}^{\frac{1-\theta_{32}}{2}}X(t)^{\frac{1+\theta_{32}}{4}}Y(t),
\]
which has the form \eqref{eq:smallICbound} with
\[
\mu_8 = \frac{1-\theta_{32}}{2} = \frac{\alpha+\rho-1}{2\rho+\alpha}\in(0,1],\quad\forall\, \rho\geq0.
\]
We would like to point out that since there is no Poincar\'e inequality on $\|\partial_z\omega\|_{L^2}$, we necessarily require 
\[
 -\frac{1-\theta_{32}}{2}+\theta_{33}\geq0,\quad\Longleftrightarrow\quad
 \delta\leq \frac{\alpha(2-\alpha)}{2(2\rho+\alpha)}+\frac{-2\alpha^2+2\alpha+2}{2\alpha} = h(\rho)+\frac\alpha2.
\]

Now we collect all admissible conditions on $(\rho,\delta)$ as follows:
\[
\delta\leq f(\rho),\quad \delta\geq g^{-1}(\rho),\quad h(\rho)\leq \delta \leq h(\rho)+\frac\rho2,
\]
together with the mild assumptions
\[
\frac\alpha2\leq\delta<\alpha+\rho,\quad \frac\alpha4\leq\rho<\alpha.
\]
Figure \ref{fig2} illustrates the admissible region in the $(\rho, \delta)$-plane.
It is straightforward to verify that all the above conditions are satisfied when $(\rho, \delta)$ are chosen as in \eqref{eq:rhodelta}, for $\alpha \in (1, \alpha_0)$.
This completes the proof of Lemma \ref{lem:small}.

\begin{figure}[ht]
\includegraphics[width=.45\linewidth]{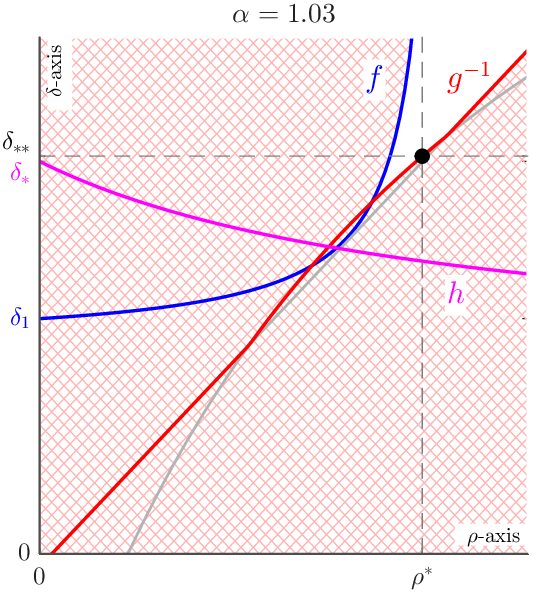}\quad	
\includegraphics[width=.45\linewidth]{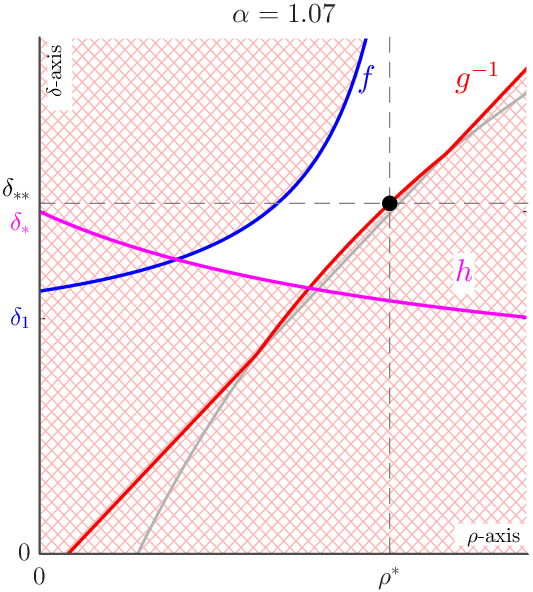}
\caption{An illustration on the admissible region, and the optimal choice of $(\rho,\delta)=(\rho^*, \delta_{**})$ .}\label{fig2}
\end{figure}

\subsection{Critical regime}
When $\alpha = 1$, from \eqref{eq:rhodelta} we have $\rho = \frac{1}{2}$ and $\delta = 1$. All estimates \eqref{eq:smallICbound} for $R_i$ from the subcritical regime carry over to the critical case, except for the term $R_3=U_{22}$, since $f(\frac12)$ is not defined.

Note that $\theta_{11} = \theta_{12} = \frac12$. From \eqref{eq:U22e}, we can readily verify that
\[
 U_{22} \leq C\|\omega_0\|_{L^\infty}^{\frac12}\|\Lambda_{h}\omega\|_{L^2}^{\frac12}\|\Lambda_{h}u\|_{L^{2}}^{\frac12} \|\Lambda_{h}^{\frac32}u\|_{L^{2}}^{\frac32}\leq C\|\omega_0\|_{L^\infty}^{\frac12}X(t)^{\frac14}Y(t).
\]
Hence, \eqref{eq:smallICbound} holds with $\mu_3 = \frac{1}{2}$.

The rest of the proof follows analogously to the subcritical case.
\vskip 3em

\bibliographystyle{plain}

\end{document}